\documentclass[12pt,a4,times]{article}
\usepackage{amsfonts}
\usepackage{amssymb}
\usepackage{amsmath, mathrsfs}
\usepackage{pdfsync, leftidx}
 \def\sqr#1#2{{\vcenter{\vbox{\hrule
        height.#2pt \hbox{\vrule width.#2pt height#1pt \kern#2pt
          \vrule width.#2pt} \hrule height.#2pt}}}} 
\newtheorem{theorem}{Theorem}[section]
\newtheorem{lemma}{Lemma}[section]
\newtheorem{example}{Example}[section]
\newtheorem{corollary}{Corollary}[section]
\newtheorem{proposition}{Proposition}[section]
\newtheorem{remark}{Remark}
\newtheorem{definition}{Definition}[section]
\newtheorem{assumption}{Assumption}

\newenvironment{proof}{{\sc Proof.}\hspace{3mm}}{\qed \\}
\newenvironment{pf-main}{{\sc Proof of Theorem \ref{mainresult}.}\hspace{3mm}}{\qed}
\newcommand{\nc}{\newcommand}
\nc{\cadlag}{c\`{a}dl\`{a}g } \nc{\caglad}{c\`{a}gl\`{a}d }
\nc{\ba}{\begin{array}} \nc{\ea}{\end{array}}
\nc{\be}{\begin{equation}} \nc{\ee}{\end{equation}}
\nc{\bea}{\begin{eqnarray}} \nc{\eea}{\end{eqnarray}}
\nc{\bean}{\begin{eqnarray*}} \nc{\eean}{\end{eqnarray*}}
\nc{\bu}{\bullet} \nc{\nn}{\nonumber} \nc{\cA}{{\mathcal A}}
\nc{\cB}{{\mathcal B}} \nc{\cC}{{\mathcal C}} \nc{\cD}{{\mathcal
D}} \nc{\cL}{{\mathcal L}} \nc{\cN}{{\mathcal
N}}\nc{\bbD}{\mathbb{D}} \nc{\cG}{{\mathcal G}} \nc{\cI}{{\mathcal I}}\nc{\cF}{{\mathcal
F}} \nc{\cS}{{\mathcal S}} \nc{\cO}{{\mathcal O}}\nc{\cR}{{\mathcal
R}}\nc{\cU}{{\mathcal U}} \nc{\cH}{{\mathcal H}}
\nc{\cK}{{\mathcal K}} \nc{\cM}{{\mathcal M}} \nc{\cP}{{\mathcal
P}} \nc{\bbE}{\mathbb{E}} \nc{\bbEQ}{\mathbb{E}^{\mathbb{Q}}}
\nc{\eps}{\varepsilon}\nc{\bbU}{\mathbb{U}}
\nc{\bbEP}{\mathbb{E}_{\mathbb{P}}}\nc{\bbL}{\mathbb{L}}
\nc{\bbP}{\mathbb{P}} \nc{\bbQ}{\mathbb{Q}} \nc{\Om}{\Omega}
\nc{\om}{\omega} \nc{\bbR}{\mathbb{R}} \nc{\bbC}{\mathbb{C}}
\nc{\bfr}{\begin{flushright}} \nc{\efr}{\end{flushright}}
\nc{\dXt}{\Delta X_{t}} \nc{\dXs}{\Delta X_{s}}
\nc{\bs}{\blacksquare} \nc{\dX}{\Delta X} \nc{\dY}{\Delta Y}
\nc{\dnkx}{\left(X(T^{n}_{k})-X(T^{n}_{k-1})\right)}
\nc{\dom}{depth-of-the-market } \nc{\uar}{\uparrow}
\nc{\dar}{\downarrow}\nc{\rar}{\rightarrow}
\nc{\half}{\frac{1}{2}}
 \nc{\hbE}{\hat{\bbE}}

\nc{\what}{\widehat} \nc{\fhat}{\what{f}} \nc{\qed}{\hfill
$\blacksquare$} \nc {\parx}{\frac{\partial}{\partial x}} \nc
{\parw}{\frac{\partial}{\partial w}} \nc
{\parww}{\frac{\partial^2}{\partial w^2}}
\def\rar{\rightarrow}
\def\dar{\downarrow}

\nc{\chf}{\mbox{$\mathbf1$}}
\numberwithin{equation}{section}
\begin{document}

\title{On absolutely continuous compensators and nonlinear filtering equations in default risk models}
        \author{Umut \c{C}etin \\ London School of Economic and Political Science\\
        Department of Statistics\\ Columbia House\\ Houghton Street\\ London WC2A 2AE\\
{\tt u.cetin@lse.ac.uk}}
\date{\today}
 \maketitle
\begin{abstract} We discuss the pricing of defaultable assets in an
  incomplete information model where the default time is given by a
  first hitting time of an unobservable process. We show that in a
  fairly general Markov setting, the indicator function of the default
  has an absolutely continuous compensator. Given this compensator we
  then discuss the optional projection of a class of semimartingales
  onto the filtration generated by the observation process and the
  default indicator process. Available formulas for the pricing of
  defaultable assets are analyzed in this setting and some alternative
  formulas are suggested. \\
{\bf Key words:} Az\'ema supermartingale, default indicator,
absolutely continuous compensators, pricing of default risk, nonlinear
filtering, Zakai equation, Kushner-Stratonovich equation.
\end{abstract}
\section{Introduction}
The motivation of this paper comes from a special field of Finance Theory; namely,
the valuation of credit derivatives. The key problem in this field
is to determine the price of an asset subject to default. To make the
discussion more concrete, let's consider the basic
financial instrument with default risk, which is a corporate bond with
maturity $T$ that pays the owner $F$ units of a currency if the firm
does not default until time $T$. If firm defaults before time $T$,
usually there is a nonzero rebate, $R$, paid to the bond holder. Given
this basic structure, the price of the defaultable bond at time $t$ is given
by the conditional expectation
\[
\bbE[R \chf_{[\tau \leq T]}+F \chf_{[\tau > T]} |\cG_t]
\]
where $\cG$ is the {\em market's filtration} and the expectation is taken
with respect to the martingale measure chosen by the market. Default
time, $\tau$, associated to the  firm issuing the
defaultable bond is often  modeled as the first hitting time of
barrier by a stochastic process representing the firm value. Leland \cite{L} shows  under certain conditions that it is optimal for
the equity owners to liquidate the firm, and thus declare default,
when the firm value falls below a barrier.
On the other hand, the market is not able to identify the firm value
continuously in time but has only a noisy observation of it. However,
it is reasonable to assume that whether the default has occurred is directly
observed in the market. In the simplest setting the process $Y$ that
satisfies
\[
Y_t=B_t +\int_0^t b(X_s)\,ds,
\]
can be viewed as the noisy observation of the firm value with $B$
being the noise, independent of the firm value, and
$X$ is the firm's value process. Various aspects of this incomplete information issue
have been studied in the literature. We can mention Jarrow \&
Turnbull \cite{jt95}, Lando \cite{l98},
Duffie \& Singleton \cite{ds99}, Kusuoka \cite{kus}, Duffie \& Lando
\cite{dl01}, Nakagawa \cite{nak}, Bielecki \&
Rutkowski \cite{br}, \c{C}etin, et al. \cite{cjpy},  Jarrow \&
Protter \cite{jp04}, Jarrow, et al. \cite{jps}, Coculescu, et al.~\cite{coc}, and Campi \&
\c{C}etin \cite{cc}  to name a few. Frey \& Runggaldier \cite{fr08},
Frey \& Schmidt \cite{fs09a}  and
Frey \& Schmidt \cite{fs09} model credit risk from a nonlinear filtering point of view.

Valuation formulas for defaultable assets are given in different
contexts in the literature. Duffie et al.~\cite{dss96} have given a
formula that computes the price in the form of a stochastic
discounting. In case of zero-coupon defaultable bond, i.e. $R=0$ and
$F=1$ in above formulation,  the time $t$  price of this bond on the event $[\tau>t]$ is
given by
\be \label{e:dss}
J_t -\bbE\left[\chf_{[t<\tau\leq T]}\Delta
  J_{\tau}|\cG_t\right], \qquad t \in [0,T],
\ee
where
\[
J_t=\bbE\left[\exp\left(-\int_t^{T\wedge \tau}\lambda_s\,ds\right)\bigg|\cG_t\right],
\]
and $\lambda$ is the so called {\em default intensity} which appears in the
$\cG$-canonical decomposition of the supermartingale $(\chf_{[\tau>t]})_{t \geq 0}$. More precisely, $(\lambda_{t \wedge \tau})_{t \geq 0}$ is what makes
\be \label{e:acc}
\left(\chf_{[\tau>t]}+\int_0^{\tau \wedge t}\lambda_s\,ds\right)_{t
  \geq 0}
\ee
 a $\cG$-martingale and $(\int_0^{\tau \wedge t}\lambda_s\,ds)_{t \geq 0}$ is said to be the {\em compensator} of $(\chf_{[\tau>t]})_{t \geq 0}$. It is important to note here that such $\lambda$ may not exist for any given random time $\tau$. Although the formula in (\ref{e:dss}) is
appealing in the sense that  the price is  a discounted expected value
where the discounting factor is given by the default intensity, its
drawback lies in the difficulty of computing the second term in (\ref{e:dss}) even if one is content with the assumption for the existence of an absolutely continuous compensator. In general it is not possible to compute
the conditional expectation of the jump term appearing in the formula
(see  \c{C}etin, et al.\cite{cjpy}  for a special case when this
computation is feasible). This led various authors suggest different
formulas for the pricing of defaultable bonds.

An alternative formula to (\ref{e:dss})  for the price of a zero-coupon defaultable bond before default is given by
\be
Z_t^{-1}\bbE[Z_T|\cF^Y_t],
\ee where $Z$ is the so-called {\em Az\'ema supermartingale} defined by $Z_t:=\bbP[\tau>t|\cF^Y_t]$. One should mention at
this point the works of Collin Dufresne, et al.~\cite{dhg}, Bielecki, et al. \cite{bjr}, Coculescu,
et al.~\cite{cjn} and Coculescu \& Nikeghbali \cite{cn} as good
references that are attempting to solve the valuation problem in the
general case. The papers \cite{cjn}, \cite{cn} and \cite{bjr} also contain a
discussion of several approaches to obtain the valuation formula.

The main assumption in the
formulas which compute the price as a discounted conditional
expectation in the works listed above, and in many others, is that the increasing process
$(\chf_{[\tau \leq t]})_{t \geq 0}$ has an absolutely continuous compensator
leading to the canonical decomposition described in
(\ref{e:acc}). This assumption has found widespread use in models of
credit risk due to intuitive representation of $\lambda$ as the
probability of default in the next instant (see \cite{dl01} for the
relation between $\lambda$ and credit spreads). In a recent paper,
Janson et al.~\cite{jmp} have identified a set of natural sufficient
conditions under which $(\chf_{[S \leq t]})_{t \geq 0}$ has an absolutely
continuous compensator for {\em any} totally inaccessible stopping time $S$ with respect to the natural filtration of a Markov process from a certain class.

The present paper has two main objectives. In Section 2 we show
that, under natural regularity conditions, $(\chf_{[\tau \leq t]})_{t \geq 0}$ has
an absolutely continuous $\cG$-compensator when $\tau$ is the first hitting
time of $0$ for the diffusion
\be\label{e:xmarkov}
X_t=X_0+W_t+\int_0^t a(X_s)\, ds
\ee
and the observation process is given by
\be \label{e:ymarkov}
Y_t=B_t+ \int_0^t b(s, X_s)\, ds,
\ee
where $B$ and $W$ are independent standard Brownian motions. More precisely, we
show the existence of an $\cF^Y$-adapted process $(\lambda_t)_{t \geq 0}$ such
that the process in (\ref{e:acc}) is a $\cG$-martingale, where $\cG$
is, as usual, the smallest filtration satisfying usual conditions that contains
$\cF^Y$ and make $\tau$ a stopping time. Modelling the default time as
the first hitting time of a stochastic process is desirable since it
is consistent with economic intuition that the equity owners are
likely to declare default when the firm value falls below a certain
level as we mentioned before. However, the disadvantage of this choice
when the underlying stochastic process is continuous is that the
default time becomes a predictable stopping time in the natural
filtration of the underlying so that it does not admit an
intensity. We refer the reader to the discussion in \cite{jp04} for
the problems with the default time being predictable. Our results show
that although the first hitting time of a continuous diffusion is a
predictable stopping time, if we shrink the filtration under the more reasonable assumption that the firm value can only
be observed with some noise, the default time becomes a totally
inaccessible stopping time and admit an
intensity. We will see that the finite variation part of the Doob-Meyer decomposition of $Z$ is absolutely continuous, which will in turn imply the existence of $\lambda$ leading to (\ref{e:acc}).  An explicit representation for $\lambda$ is also given. We achieve this by computing the
canonical representation of the associated {Az\'ema supermartingale}
using tools from non-linear filtering. We remark here that the results
of Janson et al.~\cite{jmp} are not applicable to yield an absolutely
continuous compensator since $\tau$ is not a totally inaccessible
stopping time, in fact it is predictable, in the natural filtration of
$X$ and it is, in general, not a stopping time with respect to the
filtration generated by $Y$. Thus, our results indicate that the existence of a default intensity requires much weaker conditions when there is only a noisy information on the fundamental processes that drive the default event. As for the pricing of defaultable securities, the existence of an absolutely continuous compensator implies that one can use the formulae in the aforementioned works which assume its existence. Moreover, at the end of Section 2, we will suggest some alternatives to the formula given in (\ref{e:dss}).

In view of the results in Section 2, we solve in Section 3 the
nonlinear filtering problem corresponding to the $\cG$-optional
projection of semimartingales. In
particular we obtain the {\em Kushner-Stratonovich equations} for the
$\cG$-conditional distribution of $X$.
As a by product, this suggests yet another alternative formula to price
defaultable bonds. Another use of the solution to this filtering
problem is that it immediately gives us the explicit semimartingale
decomposition of the price processes of defaultable assets, which are in
general not easy to compute
using, e.g.,  the formula (\ref{e:dss}) mentioned above. On the way to
the solution of the filtering problem, we also
briefly discuss a common assumption in default risk models, the
so-called {\bf H}-hypothesis, due to its connection to a certain {\em martingale representation
  result} which was essential in our proof of equations of nonlinear
filtering. As an application of the filtering equations,  the Doob-Meyer
decomposition for the value process of the rebate is calculated and the
equation of extrapolation is given.  An extension of the filtering equations
to a non-Markovian setting is also discussed at the end of Section 3.

Finally, it's worth  emphasising that
the setup considered in Section 2 and 3 and the specific filtering problem studied in Section
3 cannot be viewed within the standard class of filtering problems
with jump-diffusion observation which have been previously studied and
applied to
credit risk (see \cite{fr08} and \cite{CeCo}). In these models the
default times are the jump times of a marked point process or a
jump-diffusion and as such they are {\em totally inaccessible} stopping
times, with respect to the large filtration to which all the processes
are adapted, and admit an intensity. As a consequence, in every
shrinkage of the filtration, the default indicator processes will continue to have
absolutely continuous compensators. However, in our setup the default
time, being the first hitting time of a continuous diffusion, is a {\em predictable} stopping time in the large filtration and,
thus, does not admit an intensity. Moreover, it is not a priori clear
how much one needs to shrink the large filtration in order to make the
default time a totally inaccessible stopping time. These
considerations make it impossible to represent our filtering problem
within the framework of the
above models. Consequently, one needs to develop a different approach
and in Section 2 and 3 we follow the one that is outlined  above.

\section{Existence of an absolutely continuous compensator} \label{s:model}
Let $B$ and $W$ be two independent standard Brownian motions with $B_0=W_0=0$ defined
on $(\Om, \cF,(\cH_t)_{t\geq 0}, \bbP)$ satisfying the usual
hypotheses. All processes  in this and subsequent  sections will be
defined on this filtered probability space. Observe that $\cH$ is
allowed to be strictly larger than the filtration generated by $B$ and $W$.

Suppose $X$ is a diffusion which is a strong solution to
\be \label{e:sdeX}
X_t=X_0 + W_t + \int_0^t a(X_s)\,ds,
\ee
where $X_0>0$ is an $\cH_0$-measurable random variable with $\bbP(X_0\in dx)=\mu(dx)$ where $\mu$ is a probability measure on the Borel subsets of $(0,\infty)$.
\begin{assumption} \label{a:a} $\bbE X_0^2 < \infty$ and  the function $a:\bbR \mapsto \bbR$ satisfies the following:
\begin{enumerate}
\item $a$ is continuously differentiable with a bounded derivative.
\item $\lim_{x \rar \infty}A(x)$ exists, possibly infinite, where
\[
A(x):=\int_0^{x}a(y)\,dy.
\]
\item $a(\infty):=\lim_{x \rar \infty}a(x)$ exists (possibly infinite).
If $a(\infty)=-\infty$ then there exists some $K_a>0$ and $g_a \geq 0$ such that for any $x\geq g_a$
 \[
 a(x)=-K_a x + f_a(x)
 \]
 where $f_a$ is a negative function such that $-\int_0^x f_a(y)\,dy
 \leq c_f x^p$ for some $p <2$.
\end{enumerate}
\end{assumption}
\begin{remark} \label{r:2ndmomentX} Under Assumption \ref{a:a} there
  exists a unique strong solution to (\ref{e:sdeX}) such that for
  every $T>0$, $\bbE X_t^2 \leq \gamma (1+ \bbE X_0^2)e^{\gamma t}$ for all $t \in [0,T]$ for some constant $\gamma$ that depends only on $T$ and the upper bound on the derivative of $a$ (see Theorem 5.2.9 in \cite{ks}).
\end{remark}
\begin{remark} The assumption on the asymptotic behavior of $a$ is to
  ensure that $a$ does not make unbounded oscillations when it
  diverges to $-\infty$. This will be used in obtaining bounds on the
  density of the first hitting time of $0$ by $X$ below. Note that
  this assumption is satisfied when $X$ is a Gaussian process,
  i.e. when $a$ is affine.
\end{remark}
Let
\[ \tau:=\inf\{t>0: X_t = 0\}\] and define
\be \label{d:H}
H^{a}(t,x):=P_x[\tau>t],
\ee
where $P_x$ is the law of the solutions of (\ref{e:sdeX}) with $X_0=x$. Observe that $\tau$ is a predictable $\cH$-stopping time. As such, the $\cH$-compensator of the process $(\chf_{[\tau>t]})_{t \geq 0}$ is the process itself. Our main goal in this section is to show that under a particular shrinkage of the filtration, this process will have an absolutely continuous compensator.

We will show in the theorem below that
\[
H^a(t,x)=1-\int_0^t\ell^a(u, x)\,du,
\]
for some function $\ell^a$ along with some further properties of the density which will be useful in the sequel for the existence of an absolutely continuous compensator. Recall that when $a\equiv0$
\[
\ell^a(t,x)=\ell(t,x):=\frac{x}{\sqrt{ 2 \pi t^3}}\exp\left(-\frac{x^2}{2
t}\right),
\]
for $x>0$, which  is the probability density function of the first hitting time of $0$ for a standard Brownian motion started at $x$. We will also drop the superscript in $H^a$ when $a \equiv 0$ for notational convenience.
\begin{theorem} \label{t:Hdensity}  Let $H^a$ be as in (\ref{d:H}). Then, under Assumption \ref{a:a},
\begin{enumerate}
\item $H^a$ is absolutely continuous. That is, there exists a function $\ell^a$ such that
\[
H^a(t,x)=1-\int_0^t\ell^a(u, x)\,du, \qquad \forall x>0.
\]
Moreover, $H^a(t,x)>0$ for all $t \geq 0$ and $x>0$.
\item Let
\[
\delta:=\sup_{x \in \bbR_+}\frac{x}{e^{\frac{x}{6}}-e^{-\frac{5x}{6}}}
\]
and $K_g$ be the smallest constant, $K$, for which $|a(x)|\leq K(1+|x|)$ for all $x \in \bbR$. Then,
\be \label{e:tauinv}
\int_0^{\infty} \frac{1}{s}\ell^a(s,x) \,ds \leq 2 \delta^{3/2} \frac{1+ K_g x}{x^2}.
\ee
\item The mapping $t \mapsto t \ell^a(t,x)$ is locally bounded uniformly in $x$.
\end{enumerate}
\end{theorem}
\begin{proof} See Appendix.
\end{proof}

In addition to $X$ there is also an {\em observation process} $Y$
which is defined by
\be \label{e:sdeY}
Y_t= B_t + \int_0^t b(s, X_s)\, ds
\ee
 and $b: \bbR_+ \times \bbR^2\mapsto \bbR$ is satisfying the
 following:
\begin{assumption} \label{a:b} $b(t,0)=0$ for all $t \geq 0$. Moreover, $b$ is
  locally Lipschitz in $x$, thus, for every $T>0$ there exists a
  $K_b(T)$ such that $|b(t,x)| \leq K_b(T) |x|$ for all $t \leq T$.
\end{assumption}
Note that the assumption $b(t,0)=0$ for all $t$ is without loss of
generality since the filtrations generated by $Y$ or $Y-
\int_0^{\cdot}b(s,0)\, ds$ are the same.

In this section we are mainly interested in the {\em
Az\'ema supermartingale}
\be \label{d:Z}
Z_t:=\bbP[\tau
> t|\cF^Y_t]
\ee
 where $\cF^Y$ is the minimal filtration satisfying the usual
 conditions generated by $Y$. As the conditional expectation is only
 defined almost surely for each $t$, $Z$ is defined to be the unique
 $\cF^Y$-\emph{optional projection} of $(\chf_{[\tau>t]})_{t \geq 0}$. We recall the definition of optional projection here for the convenience of the reader.
\begin{definition}\label{d:OP} Let $U$ be a positive or bounded measurable process and $(\cF_t)$ be a filtration satisfying the usual conditions. Then, the $(\cF_t)$-optional projection of $U$ is the $(\cF_t)$-optional process $V$ such that for any $\cF$-stopping time $S$ the following holds:
\[
\bbE[U_S\chf_{[S<\infty]}|\cF_S]=V_S\chf_{[S<\infty]}.
\]
\end{definition}
The above definition, taken from Section 5 in Chap. IV of \cite{RY}, has an obvious extension
to integrable measurable processes.  We emphasize here that this
choice of optional projection will be made without notice
whenever we consider processes defined by projection onto a smaller filtration, in particular when we consider the filtering of a signal by an observation process.

 $Z$, being a $(\bbP, \cF^Y)$-supermartingale, has a \cadlag
modification  due to the
continuity of the map $t \mapsto \bbP[\tau>t]$,  (see Theorem 2.9 in Chap.~II of \cite{RY}), which we will use
henceforth. Note that $Z$ is a
nonnegative supermartingale of class $D$.
The Doob-Meyer decomposition for such
supermartingales (see Theorem 8 in
Chap.~III of \cite{Pro}) gives the following.
\begin{proposition} \label{p:Z-DM} There exists a unique
  increasing and $\cF^Y$-predictable process $C$ with $C_0=0$ such
  that $Z+C$ is a uniformly integrable
  $(\bbP,\cF^Y)$-martingale.
\end{proposition}
In the rest of this section we will compute the above decomposition
explicitly and discuss some of its consequences. The following is the main result of this section
whose lengthy proof is delegated to the appendix.
\begin{theorem} \label{t:AC} Let $Z$ be the Az\'ema supermartingale given by
  (\ref{d:Z}),  and $C$ be as in Proposition \ref{p:Z-DM}. Then,
  under Assumption \ref{a:a} and
  \ref{a:b} the following holds:
\begin{enumerate}
\item $Z$ is a.s. strictly
positive and  for any $ t \geq 0$
\bea \label{e:Zfilter}
Z_t&=&\bbE[H^a(t, X_0)] \\
&&+\int_0^t
\bbE\left[\chf_{[\tau>s]}H^a(t-s,X_s)\left(b(s,X_s)-\bbE[b(s,X_s)|\cF^Y_s]\right)\big|\cF^Y_s\right]\,
dB^Y_s, \nn
\eea
where \[
B^Y_t=Y_t-\int_0^t\bbE[b(s,X_s)|\cF^Y_s]\,ds
\]
is an $\cF^Y$-Brownian motion.
\item $C_t=\int_0^t  c_s\, ds$, where
\[
c_t=\int_0^{\infty} \ell^a(t,x)\mu(dx)+ \int_0^t \bbE[\chf_{[\tau>s]}\ell^a(t-s, X_s)\left(b(s,X_s)-\bbE[b(s,X_s)|\cF^Y_s]\right)|\cF^Y_s]dB^Y_s,
\]
and $\mu$ corresponds to the initial distribution of $X_0$.
\item $Z+C$ is a uniformly integrable $(\bbP, \cF^Y)$-martingale defined by $ Z_t+C_t=\int_0^t \eta_s\, dB^Y_s$, where
\[
\eta_t:=\bbE[\chf_{[\tau>t]}b(t,X_t)|\cF^Y_t]-Z_t\bbE[b(t,X_t)|\cF^Y_t].
\].
\end{enumerate}
\end{theorem}
\begin{proof}
See the Appendix.
\end{proof}

The next remark is considering a possible relaxation of the
independence assumption on $B$ and $W$. However, as it heavily relies
on a certain argument in the proof of Theorem \ref{t:AC},  the reader
is invited to read the following remark along with the proof of the preceding theorem.

\begin{remark} \label{r:independence} A natural question at this point is `how important is
  the independence assumption on $B$ and $W$?' To this end let's
  suppose $X=X_0+ W$  and $d[B,W]_t=\varrho_t\,dt$ where $\varrho$ is a
  progressively measurable process and $X_0$ is a strictly
  positive constant. Repeating what we did in the proof
  of Theorem \ref{t:AC} yields
\bean
Z_t&=&H^a(t, X_0) +\int_0^t \bbE[\chf_{[\tau>s]}H_x(t-s, X_s)\varrho_s|\cF^Y_s]\,dB^Y_s\\
&&+\int_0^t\Big\{
\bbE\left[\chf_{[\tau>s]}H(t-s,X_s)\left(b(s,X_s)-\bbE[b(s,X_s)|\cF^Y_s]\right)\big|\cF^Y_s\right]\Big\}\,
dB^Y_s. \nn
\eean
Thus, in order to arrive at a similar decomposition we obtained in Parts 2 and 3 of Theorem \ref{t:AC}, we will need some assumptions on the correlation
coefficient $\varrho$. Indeed, if $B$ and $W$ are the same Brownian
motions, i.e. $\varrho\equiv 1$, it is clear that $Z=\chf_{[\tau>t]}$ and
its compensator is itself. On the other hand if $|\varrho_t| \leq \varrho
|X|_t^2$ for some constant $\varrho \geq 0$ (at least when $X$ is within
some open interval including $0$),
then
\[
\chf_{[\tau>s]}|\varrho_s \ell_x(t-s,X_s)|\leq
\chf_{[\tau>s]}\varrho \, \ell(t-s,X_s)\left(X_s+\frac{X^3_s}{t-s}\right) .
\]
Thus, using the explicit form of $\ell$, one can repeat the arguments
that led to the explicit Doob-Meyer decomposition in Theorem \ref{t:AC}
to justify the interchange of ordinary and stochastic integrals and
establish that $C$ is absolutely continuous. Observe that by placing
this assumption on $\varrho$, what we in fact require is that the correlation
coefficient between two Brownian motions is vanishing quite
fast when $X$ is approaching to $0$, i.e. $B$ and $W$ are behaving
almost independently when $X$ is in a neighborhood of $0$.  It would
be interesting to investigate whether such a condition is a necessary
condition for $C$ to be absolutely continuous.
\end{remark}

In the Introduction, we claimed that the absolute continuity of $C$ would lead to $(\chf_{[\tau>t]})_{t \geq 0}$ having an absolutely continuous compensator. We are now in a position to make this precise and show that it is indeed the case. To this end, let $\cG=(\cG_t)_{t \geq 0}$
be the filtration generated by $D$ and $Y$, and augmented with the $\bbP$-null sets, where  $D_t:=\chf_{[\tau >t]}$. Then, $D$ is a
$\cG$-adapted \cadlag  $\bbP$-supermartingale and there exists a
$\cG$-predictable $\Lambda$ with $\Lambda_0=0$ such that $D +\Lambda$
is a $(\bbP, \cG)$-martingale. The $\cF^Y$-decomposition of $Z$ allows us to compute $\Lambda$ directly as follows:
\begin{corollary} \label{c:lambda} Under the assumptions of Theorem  \ref{t:AC},  $D+\Lambda$ is a $(\bbP, \cG)$-martingale such that $d\Lambda_t= \chf_{[\tau \geq
t]}\lambda_t dt$ where
\[
\lambda_t=\frac{\int_0^{\infty} \ell^a(t,x)\mu(dx) +\int_0^t \bbE[\chf_{[\tau>s]}\ell^a(t-s, X_s)\left(b(s,X_s)-\bbE[b(s,X_s)|\cF^Y_s]\right)|\cF^Y_s]dB^Y_s}{\int_0^{\infty} H^a(t,x)\mu(dx)+\int_0^t \bbE[\chf_{[\tau>s]}H^a(t-s, X_s)\left(b(s,X_s)-\bbE[b(s,X_s)|\cF^Y_s]\right)|\cF^Y_s]dB^Y_s},
\]
and $\mu$ is the probability distribution of $X_0$.
\end{corollary}
\begin{proof} It is well-known (see, e.g., Theorem 3.4 in \cite{cjn}) that
\[
\lambda_t=\frac{1}{Z_{t-}}\frac{dC_t}{dt}.
\]
The result now follows from Theorem \ref{t:AC}, and that
\[
\bbE[H^a(t,X_0)]= \int_0^{\infty} H^a(t,x)\mu(dx).
\]
\end{proof}\\
Given the above formulation of $Z$ we have the following representation formula as a consequence of, e.g., Proposition 2.3 in Chap.~IX of \cite{RY}.
\begin{corollary} \label{c:Zrep} Under the assumptions of Theorem  \ref{t:AC},
\bean
Z_t&=&\exp\left(-\int_0^t\lambda_s\,ds\right) \xi_t^{-1}\kappa_t,
\mbox{ where}\\
\xi_t&=&\exp\left(\int_0^t\bbE[b(s,X_s)|\cF^Y_s]\,dY_s-\frac{1}{2}\int_0^t\bbE^2[b(s,X_s)|\cF^Y_s]\,ds\right),
\\
\kappa_t&=&\exp\left(\int_0^t
  \frac{\bbE[\chf_{[\tau>s]}b(s,X_s)|F^Y_s]}{Z_s}dY_s-\frac{1}{2}\int_0^t
  \frac{\bbE^2[\chf_{[\tau>s]}b(s,X_s)|F^Y_s]}{Z^2_s}\,ds\right).
\eean
\end{corollary}
\begin{proof} Note that
\[
dZ_t=\bbE[\chf_{[\tau>t]}b(t,X_t)|\cF^Y_t]dB^Y_t-Z_t(\lambda_t dt
+\bbE[b(t,X_t)|\cF^Y_t]dB^Y_t).
\]
Thus, it follows from  Proposition 2.3 in Chap.~IX of \cite{RY} that
\[
Z_t=\xi_t^{-1}\exp\left(-\int_0^t\lambda_s\,ds\right)\left(1+\int_0^t \xi_s\exp\left( \int_0^s\lambda_r\,dr\right)\bbE[\chf_{[\tau>s]}b(s,X_s)|F^Y_s]\,dY_s\right),\]
where
\[
\xi_t=\exp\left(\int_0^t\bbE[b(s,X_s)|\cF^Y_s]\,dY_s-\frac{1}{2}\int_0^t\bbE^2[b(s,X_s)|\cF^Y_s]\,ds\right).
\]
Let
\[
\kappa_t:=\left(1+\int_0^t \xi_s\exp\left(
    \int_0^s\lambda_r\,dr\right)\bbE[\chf_{[\tau>s]}b(s,X_s)|F^Y_s]\,dY_s\right)
\]
and observe  that $\kappa$ is strictly positive with
$d\kappa_t=\frac{\kappa_t}{Z_t}\bbE[\chf_{[\tau>t]}b(t,X_t)|F^Y_t]\,dY_t$,
i.e.
\[
\kappa_t=\exp\left(\int_0^t
  \frac{\bbE[\chf_{[\tau>s]}b(s,X_s)|F^Y_s]}{Z_s}dY_s-\frac{1}{2}\int_0^t
  \frac{\bbE^2[\chf_{[\tau>s]}b(s,X_s)|F^Y_s]}{Z^2_s}\,ds\right).
\]
\end{proof}
\begin{remark}The above corollary also gives the multiplicative
  decomposition of $Z$ as a product of a local martingale and a
  decreasing process. Indeed, it is a straightforward application of
  integration by parts formula to see that $\xi^{-1} \kappa$ is a
  $(\bbP, \cF^Y)$-local martingale. Observe that one can obtain the
  multiplicative decomposition directly from the Doob-Meyer
  decomposition of $Z$ since $Z=n e^{-\int_0^{\cdot}\lambda_s\,ds}$ where $dn=e^{\int_0^{\cdot}\lambda_s\,ds}(dZ+dC)$ with
  $n_0=1$, and
  $C$ is as defined in Proposition \ref{p:Z-DM}.
\end{remark}
We now will take a detailed look at the formula in
(\ref{e:dss}). Recall that the expression in (\ref{e:dss}) equals
\be \label{e:S}
S_t:=\bbP[\tau>T|\cG_t]
\ee
 on the set $[\tau>t]$.

Our aim in the rest of this section is to obtain alternative
representations for $S$  which will emphasize the role of default intensity
as a stochastic discount factor as observed in the Introduction. These
representations will be obtained via equivalent changes of probability
measure and the our first change of measure will be defined by the
process $M$ given by
\be \label{e:M}
M_t:=\exp\left(\int_0^t b(s,X_s)dY_s -\frac{1}{2}\int_0^T
  b^2(s,X_s)\, ds\right).
\ee
Observe that
\[
dM^{-1}_t=M^{-1}_t b(t,X_t)\,dB_t,
\]
and, thus, $M^{-1}$ is a strictly positive $(\bbP, \cH)$-martingale due to the fact that $(X,Y)$ is
a non-explosive solution to (\ref{e:sdeX}) and (\ref{e:sdeY}), see,
e.g., Exercise 2.10 in Chap.~IX of \cite{RY}. Therefore,  for each
$t>0$ one can define a probability measure  $\bbQ_t$ on $\cH_t$ such that
\[
\frac{d\bbQ_t}{d\bbP_t}=M_t^{-1},
\]
where $\bbP_t$ is the restriction of $\bbP$ to $\cH_t$. Under
$\bbQ_t$, $(Y_s)_{s \in [0,t]}$ is a standard Brownian motion
independent of $(X_s)_{s \in [0,t]}$. The reason  for defining a
family of probability measures rather than a single $\bbQ$ valid on
$\cH_{\infty}$ is due to the fact that $M^{-1}$ is not necessarily a
uniformly integrable martingale in this infinite horizon setting. Nevertheless, for notational convenience we
will drop the subscript in $\bbQ_t$ and write $\bbQ$ in what follows when no confusion
arises.

A useful observation, which we will often make use of in the sequel, is
that for any integrable and $\cH_t$-measurable random variable $F$
one has
\be \label{e:bayfor}
\bbE[F|\cF^Y_t]=\frac{\bbE^{\bbQ}[F
  M_t|\cF^Y_t]}{\bbE^{\bbQ}[M_t|\cF^Y_t]}.
\ee
In particular, taking $F=M_t^{-1}$ yields
\be \label{e:minv}
\bbE[M_t^{-1}|\cF^Y_t]=\frac{1}{\bbE^{\bbQ}[ M_t|\cF^Y_t]},
\ee
and $F=\chf_{[\tau>t]}$ gives
\be \label{e:ZQ}
Z_t=\frac{\bbE^{\bbQ}[ \chf_{[\tau>t]}M_t|\cF^Y_t]}{\bbE^{\bbQ}[
  M_t|\cF^Y_t]}.
\ee
The next lemma is folklore in Stochastic
  Filtering Theory  and would have followed from
  Theorem 8.1 in \cite{ls} if $M$ were a square integrable $(\bbQ, \cH)$-martingale. Since the standard texts on Filtering Theory does
  not appear to be giving the proof for the general case, we nevertheless provide its proof in the Appendix.

\begin{lemma} \label{l:RN} Under Assumption \ref{a:a} and \ref{a:b} we have
\bean
\xi_t=\bbEQ[M_t|\cF^Y_t]&=&1+\int_0^t
\bbEQ[M_s|\cF^Y_s]\bbE[b(s,X_s)|\cF^Y_s]\,dY_s, \\
\xi_t^{-1}=\bbE[M_t^{-1}|\cF^Y_t]&=&1-\int_0^t
\bbE[M_s^{-1}|\cF^Y_s]\bbE[b(s,X_s)|\cF^Y_s]\,dB^Y_s.
\eean
\end{lemma}

In view of the above lemma we see that $\xi^{-1}$ is a
$(\bbP,\cF^Y)$-martingale. This leads to the following
\begin{corollary} \label{c:dbp} Let $S$ be as in (\ref{e:S}). Under
  Assumption \ref{a:a}  and \ref{a:b}
\[
S_t=\chf_{[\tau>t]}\bbE^{\bbQ}\left[\exp\left(-\int_t^T\lambda_s\,ds\right)\exp\left(\int_t^T \vartheta_s
  dY_s-\frac{1}{2}\int_t^T \vartheta_s^2\,ds\right)\bigg|\cF^Y_t\right],
\]
where $\vartheta_s=\frac{\bbE\left[\chf_{[\tau>s]}b(s,X_s)|F^Y_s\right]}{Z_s}$.
\end{corollary}
\begin{proof} It is well-known that (see, e.g., Lemma 3.1 in \cite{ejy})
\[
\bbE[\chf_{[\tau>T]}|\cG_t]=\chf_{[\tau>t]}\frac{\bbE[\chf_{[\tau>T]}|\cF^Y_t]}{Z_t}=\chf_{[\tau>t]}\frac{\bbE[Z_T|\cF^Y_t]}{Z_t}.
\]
Since
\[
\bbE[Z_T|\cF^Y_t]=\frac{\bbE^{\bbQ}[Z_TM_T|\cF^Y_t]}{\bbE^{\bbQ}[M_t|\cF^Y_t]},
\]
the claim follows from the representation in Corollary \ref{c:Zrep}
and Lemma \ref{l:RN}.
\end{proof}
\begin{remark} Similarly, one can get the following formula for any $F \in
L^1(\cF^Y_T, \bbP)$:
\[
\chf_{[\tau>t]}\bbE[F\chf_{[\tau>T]}|\cG_t]=\chf_{[\tau>t]}\bbE^{\bbQ}\left[F\exp\left(-\int_t^T\lambda_s\,ds\right)\exp\left(\int_t^T \vartheta_s
  dY_s-\frac{1}{2}\int_t^T \vartheta_s^2\,ds\right)\bigg|\cF^Y_t\right].
\]
\end{remark}
The above corollary can be viewed as an alternative to the formula in
(\ref{e:dss}). One advantage is that it does not require a computation
of a jump term, in addition to  the conditional expectation being taken with respect
to arguably a simpler filtration, $\cF^Y$. The price to pay in return
is that the computation is made under a different, but equivalent,
probability measure and $S$ is not equal to the conditional
expectation of $\exp\left(-\int_t^T\lambda_s\,ds\right)$ but that of
its multiplication by a strictly positive deflator,
$\kappa_T/\kappa_t$. Observe that $\kappa$ is a strictly positive
$(\bbQ,\cF^Y)$-local martingale. In case $\kappa$ is a true martingale, we can
make another change of probability measure and obtain the following
result, which is a version of Proposition 4.3 in \cite{cn}.
\begin{corollary} \label{c:cn} Suppose that $(\kappa_t)_{t \in [0,T]}$ is a $(\bbQ,
  \cF^Y)$-martingale and define $\tilde{\bbQ}$ on $\cF^Y_T$ by setting
  $\frac{d \tilde{\bbQ}}{d\bbQ}=\kappa_T$ on $\cF^Y_T$. Then, under
  Assumption \ref{a:a} and \ref{a:b},
\[
S_t=\chf_{[\tau>t]}\bbE^{\tilde{\bbQ}}\left[\exp\left(-\int_t^T\lambda_s\,ds\right)\bigg|\cF^Y_t\right].
\]
\end{corollary}
Note that $\bbP\sim \tilde{\bbQ}$, too. As observed by \cite{cn} the above formula is in the spirit of the
pricing formula of \cite{dhg}, who have obtained a pricing formula as
an expectation of  $\exp\left(-\int_t^T\lambda_s\,ds\right)$ but with respect to a
probability measure which is only absolutely continuous with respect
to $\bbP$. We refer the reader to the example in $\cite{cn}$ that
illustrates the difficulty with computing that expectation.
\begin{remark} A sufficient condition for $\kappa$ being a
  $(\bbQ, \cF^Y)$-martingale is the boundedness of $b$. Indeed, under this
  assumption
\[
\frac{\left|\bbE\left[\chf_{[\tau>s]}b(s,X_s)|F^Y_s\right]\right|}{Z_s}\leq
\frac{\bbE\left[\chf_{[\tau>s]}|b(s,X_s)||F^Y_s\right]}{Z_s}\leq
K\frac{\bbE\left[\chf_{[\tau>s]}|F^Y_s\right]}{Z_s}=K,
\]
where $K$ is an upper bound on $|b|$.
\end{remark}
\begin{remark} The formulae given above, in particular the
  expression for $\lambda$,  often contains conditional
  expectations of the form $\bbE[\chf_{[\tau>t]}F(X_t)|\cF^Y_t]$ where
  $F$ is a smooth function vanishing at $0$. In general, it is not
  possible to compute such expectations since it is not possible to
  solve for the conditional distribution of $X$ analytically. However,
  there are certain numerical methods that can be used to calculate
  these values. If we let
\[
\rho_t F:=\bbE^{\bbQ}[M_t F(X_{t\wedge \tau})|\cF^Y_t],
\]
then it follows that
\[
\bbE[\chf_{[\tau>t]}F(X_t)|\cF^Y_t]=\frac{\rho_t F}{\rho_t \chf},
\]
where $\chf$ is the constant function that takes the value $1$,
whenever $F$ is a smooth function vanishing at $0$. Since the
infinitesimal generator of $X$ and $X^{\tau}$ are the same, the standard arguments from nonlinear
  filtering yield the {\em Zakai equation}
\be \label{e:zaknum}
\rho_t F=\rho_0 F +\int_0^t \rho_s \mathcal{A} F \, ds +\int_0^t \rho_s b F\,
dY_s.
\ee
where
\[
\mathcal{A}=a(x)\frac{d}{dx}+\frac{1}{2}\frac{d^2}{dx^2}.
\]
Numerical solution of (\ref{e:zaknum}) is beyond the scope of this
paper. However, the {\em splitting-up} and {\em particle} methods
which have been studied extensively in the literature can be applied
to this setting to solve  (\ref{e:zaknum}) numerically. The reader can
find a lengthy discussion of these methods in Chapter 8 and 9 of
\cite{bc}  and the references therein.
\end{remark}
Observe that although one can find the price of a defaultable asset by
the formulas suggested above, they do not give immediately the
$\cG$-semimartingale decomposition of the price process. This is going
to be the
subject of the next section where we discuss the canonical
decomposition of $\cG$-optional projections of a class of
$\cH$-semimartingales. In particular we obtain the $\cG$-canonical
decomposition of $S$ in (\ref{e:condprob}).
\section{Nonlinear filtering equations for partially observable processes}
In this section we will investigate the `optimal filters' of $\cH$-adapted \cadlag processes when the available information is generated by the processes $Y$ and $D$. As opposed to the previous section, we will restrict our attention to a finite horizon $T$. Recall from the previous section that $D_t=\chf_{[\tau>t]}$, and we now set $\cG=(\cG_t)_{t \in [0,T]}$ (resp.~$\cF^Y=(\cF^Y_t)_{t \in [0,T]})$
to be  the filtration generated by $(D_t)_{t \in [0,T]}$ and $(Y_t)_{t
  \in [0,T]}$ (resp.~$(Y_t)_{t \in [0,T]}$ only), and augmented with
the $\bbP$-null sets.  Let $\bbQ\sim \bbP_T$ be a
probability measure on $(\Om, \cH_T)$ defined by the Radon-Nikodym
density $\frac{d\bbP_T}{d\bbQ}=M_T$, where $\bbP_T$ is the restriction of $\bbP$ to $\cH_T$ and $M$ is as in (\ref{e:M}). It follows that  $(Y_t)_{t \in
    [0,T]}$ is a $\bbQ$-Brownian motion independent of $(X_t)_{t \in
    [0,T]}$.  This  in particular implies that the natural filtration
  of $Y$ is right continuous when augmented with $\bbP$-null sets and, thus, all $(\bbP,\cG)$ (resp.~$(\bbP, \cF^Y)$) martingales have  right continuous versions, which we will use henceforth.

Note that  in view of Theorem \ref{t:AC} and its Corollary \ref{c:lambda} from
Section \ref{s:model}
\be \label{d:L}
L_t:=D_t -1 +\int_0^{t \wedge \tau}\lambda_s\,ds
\ee
defines a $(\bbP,\cG)$-martingale with a single jump at $\tau$ of
size $-1$. {\em In the rest of this section, we will assume that
  Assumption \ref{a:a} and \ref{a:b}, which yield in
  particular Theorem \ref{t:AC} and its Corollary \ref{c:lambda}, are in force without an explicit mention.}

We will obtain the filtering equations via an {\em innovations
  approach} (see Kallianpur \cite{Kal} or Liptser \& Shiryaev
\cite{ls} for the background). To do this we need to obtain a
martingale representation result for the square integrable $(\bbP,\cG)$-martingales. We will soon see that all such martingales can be
written as a stochastic integral with respect to some $\cG$-Brownian
motion and $L$.  The following is a well-known result in Filtering Theory.
\begin{proposition} Let
\[
\beta_t=Y_t -\int_0^t \bbE[b(s,X_s)|\cG_s]\, ds, \qquad t \in [0,T].
\]
Then, $\beta$ is a $(\cG, \bbP)$-Brownian motion.
\end{proposition}

We will next obtain a stochastic integral representation for the martingales in $\mathscr{M}$, where
\be \label{e:sM}
\mathscr{M}=\{U: U \mbox{ is a square integrable $(\bbP, \cG)$-martingale}\}.
 \ee

In credit risk models it is often assumed, in order to simplify the
computations, that the following assumption, called {\bf H}-Hypothesis, holds:
$$
\mbox{Every $(\bbP,\cF^Y)$-martingale is a
$(\bbP,\cG)$-martingale}. \leqno(\mathbf{H})
$$
This assumption in particular implies a martingale
representation property for square integrable $\cG$-martingales, see
\cite{kus}. The following is a well-known result taken from \cite{cjn}.
\begin{theorem} \label{t:DM} Every $(\bbP,\cF^Y)$-martingale is a
$(\bbP, \cG)$-martingale if and only if
\[
\bbP(\tau \leq s |\cF^Y_t)=\bbP(\tau \leq s |\cF_{T}^Y),
\]
for every $s \leq t \leq T$.
\end{theorem}
 It is not difficult to see that  {\bf H}-Hypothesis is not
 satisfied in general in our setting since $Z$ has a non-zero
 martingale part in its canonical decomposition (see Theorem 3.3 in
 \cite{cjn}).  Nevertheless, we will  have a predictable
 representation result for the martingales in $\cM$ in the absence of
 this hypothesis in Proposition \ref{p:MRP}. Before the statement and
 a short
 proof of this result, we will prove a  proposition which will show
 that the {\bf H}-Hypothesis  is
 satisfied (locally) under an equivalent probability measure as an
 aside. To this end, let's introduce the
 positive supermartingale
\[
N_t=1- \int_0^t N_s \bbE[b(s,X_s)|\cG_s]\,d\beta_s, \qquad t \in[0,T].
\]
Let $R_n:=\inf\{t>0:N_t>n\mbox{ or } N_t<\frac{1}{n}\}$ with the convention that
$\inf\emptyset=T$. Note that since $N$ is continuous, and strictly
positive  due to $\int_0^T \bbE[b^2(s,X_s)]\,ds<\infty$, $R_n\uar
T, \bbP$-a.s.. Associated to this stopping time, let $\cF^n$ be the
filtration generated by $Y^{R_n}$, augmented with the $\bbP$-null sets,
and $\cG^n$ be the smallest filtration containing $\cF^n$ with respect
to which $\tau$ is a stopping time.
\begin{proposition} \label{p:HHyp} Let $\bbP^{n}$ be the probability measure on $(\Om, \cG_T)$ defined by
\[
\frac{d\bbP^{n}}{d\bbP}=N_{R_n}.
\]
Then, every $(\bbP^{n}, \cF^n)$-martingale  is a $(\bbP^{n}, \cG^n)$-martingale.
\end{proposition}
\begin{proof} Observe that $[Y^{R_n},Y^{R_n}]_t=t \wedge R_n$ so that
  $R_n$ is a $\cF^n$-stopping time. Moreover, $Y^{R_n}$ becomes a Brownian motion stopped at $R^n$ under $\bbP^{n}$ while the canonical decomposition of $D$ remains unchanged, i.e.
\[
D=1+ L-\Lambda,
\]
where $L$, as defined by (\ref{d:L}), is still a martingale under
$\bbP^{n}$ and $\Lambda$ is the continuous and increasing process
defined in Corollary \ref{c:lambda}. Let $Z^{n}$ denote the
$\cF^n$-optional projection of $D$ under $\bbP^{n}$. Then, it
follows from Theorem 8.1 in \cite{ls} that
\[
Z^{n}_t =1-\int_0^t \bbE^{n}[\chf_{[\tau>s]}\lambda_s|\cF^n_s]\,ds,
\]
where $\bbE^{n}$ is expectation with respect to $\bbP^{n}$ since
$Y^{R_n}$ has no drift, and $\langle L, Y\rangle \equiv 0$. Also observe that on
$t<R_n$,
\[
Z^{n}_t =1-\int_0^t
\lambda_s\bbE^{n}[\chf_{[\tau>s]}|\cF^n_s]\,ds=1-\int_0^t \lambda_s
Z^n_s\,ds,
\]
since $\lambda$ is adapted to $\cF^Y$ in view of Corollary \ref{c:lambda}. As
seen, $Z^n$ is a continuous and decreasing process, and on
$[t<R_n]$ it is given by
\be  \label{e:Z*}
Z^n_t=\exp\left(-\int_0^t \lambda_s\, ds\right).
\ee
Moreover, it follows from Theorem 8.4 in \cite{ls} that, for $s \leq t \leq T$,
\[
\bbP^{n}[\tau>s|\cF_t]=Z^{n}_s.
\]
The result now follows from Theorem \ref{t:DM}.
\end{proof}\\
The next is the  integral representation theorem that we are after.
 \begin{proposition} \label{p:MRP} Let $\mathscr{M}$ be as in (\ref{e:sM}). For any $U \in \mathscr{M}$ there exists a pair of  $\cG$-predictable  process, $(\Phi_t)_{t \in [0,T]}$ and $(\zeta_t)_{t \in [0,T]}$ such that
\[
U_t= U_0 +\int_0^t\Phi_s\, d\beta_s +\int_0^t \zeta_s dL_s.
\]
\end{proposition}
\begin{proof} It is clear that the vector semimartingale $(Y, D)$ has
  the {\em weak predictable representation property} in the sense of
  Definition 13.13 in Chap.~XIII of \cite{HWY} for $(\bbQ, \cG)$-local
  martingales. Then it follows from Theorem 13.21 in Chap.~XIII of
  \cite{HWY}  that it has the weak predictable representation property
  for $(\bbP, \cG)$-local
  martingales as well. This implies the claimed representation.
\end{proof}
We now return to solve the filtering problem when the observation is
via the processes $Y$ and $D$. Let's suppose that the unobserved
signal, $P=(P_t)_{t \in [0,T]}$ is a $(\bbP, \cH)$-semimartingale such that
\be \label{e:P}
P_t=P_0 + \int_0^t V_s \,ds + m_t,
\ee
where $m$ is a {\em continuous} $(\bbP, \cH)$-martingale and $V$ is a
measurable stochastic process adapted to $\cH$ such that, $\bbP$-a.s.,
\[
\int_0^t |V_s| \,ds < \infty,
\]
for every $t \geq 0$. The
solution of the filtering problem amounts to finding the
semimartingale decomposition of the $(\bbP, \cG)$-optional projection of $P$, which will be denoted with $\hat{P}$. We make the following
assumption on $P$.
\begin{assumption}\label{a:P} The semimartingale $P$ in (\ref{e:P})
  satisfies the following:
\begin{enumerate}
\item $\sup_{t \leq T} \bbE[P^2_t]< \infty$;
\item $\bbE\int_0^T V_s^2 ds<\infty$.
\item $m_t =\int_0^t \theta_s dB_s +n_t$ where $\theta$ is an $\cH$-predictable process and $n$ is a continuous
  $(\bbP, \cH)$-martingale strongly orthogonal to $B$.
\end{enumerate}
\end{assumption}

In equations of nonlinear filtering, see, e.g. Theorem 8.1 in
\cite{ls}, in order to obtain the filtering equations for the signal
of the form $\int_0^{\cdot} V_s\, ds + m$, where $m$ is a
$(\bbP, \cH)$-martingale, one needs the following useful fact, proof
of which is the same as that of Lemma 8.4 in \cite{ls}, hence is omitted.
\begin{proposition} \label{p:ls} For any measurable and $\cH$-adapted process $V$ with the property
\[
 \int_0^T \bbE[V^2_s]\,ds <\infty,
\]
the random process
\[
\left(\bbE\left[\int_0^t V_s\, ds\bigg|\cG_t\right]-\int_0^t \bbE\left[V_s|\cG_s\right]\, ds\right)_{t\in [0,T]}
\]
is a square integrable $(\bbP,\cG)$-martingale.
\end{proposition}

The next theorem giving the semimartingale decomposition of $\hat{P}$
is the main result of this section.  In what follows we will write
$\bbE[P_t|\cF^Y_t,\tau=t]$ for the measurable function
$\bbE[P_t|\cF^Y_t,\tau]|_{\tau=t}$ (see Lemma \ref{l:pAd} in this
respect).

\begin{theorem} \label{t:mainfilter}  Let $P$ defined by (\ref{e:P}) satisfy Assumption
  \ref{a:P}. Then
\[
\hat{P}_t=\hat{P}_0+\int_0^t \hat{V}_s\, ds +\int_0^t\Phi_s\,
  d\beta_s+\int_0^t\zeta_s\,dL_s,
\]
where
\bean
\Phi_t&=&\hat{\theta}_t+\bbE[P_t
b(t,X_t)|\cG_t]-\hat{P}_t\bbE[b(t,X_t)|\cG_t],\mbox{ and}\\
\zeta_t&=&\hat{P}_{t-} -\bbE[P_t|\cF^Y_t,\tau=t],
\eean
for every $t \in [0,T]$. In particular, the process
$(\bbE[P_t|\cF^Y_t,\tau=t])_{t \in [0,T]}$ is $\cG$-predictable.
\end{theorem}
\begin{proof}
It  follows from Proposition \ref{p:ls} that
\[
\left(\bbE\left[\int_0^t V_s\, ds\bigg|\cG_t\right]-\int_0^t \bbE\left[V_s|\cG_s\right]\, ds\right)_{t\in [0,T]}
\]
is a square integrable $(\bbP,\cG)$-martingale. Thus,
\[
\left(\hat{P}_t-\int_0^t \hat{V}_s\, ds\right)_{t\in [0,T]}
\]
is a square integrable $(\bbP,\cG)$-martingale under Assumption \ref{e:P}, and, in
view of Proposition \ref{p:MRP}, there exist $\cG$-predictable processes,
$\Phi$ and $\eta$, such that
\[
\hat{P}_t=\hat{P}_0+\int_0^t \hat{V}_s\, ds+\int_0^t\Phi_s\,
  d\beta_s+\int_0^t\zeta_s\,dL_s.
\]
So, it remains to determine the processes $\Phi$ and $\zeta$. Let
$Y^n_t:=Y_{t \wedge S_n}$ where $S_n:=\inf\{t>0:|Y_t|>n\}.$ First
note that, using integration by parts formula,
\be \label{e:op1}
Y^n_t \hat{P}_t=\int_0^t \left\{Y^n_s \hat{V}_s+\chf_{[s\leq S_n]}(\Phi_s+\hat{P}_s\bbE[b(s,X_s)|\cG_s])\right\}\,
ds +n^1_t,
\ee
where $n^1$ is a $(\bbP,\cG)$-local martingale. We will next compute the
optional projection of $Y^n P$ by directly taking the projection of
\[
Y^n_t P_t=\int_0^t \left\{Y^n_s V_s+\chf_{[s\leq S_n]} P_s b(s,X_s)\right\} \,ds  + \int_0^t
Y^n_{s}\,dm_s +\int_0^t \chf_{[s\leq S_n]} P_s\, dB_s + [m,
B]_{t\wedge S_n}.
\]
Thus,
\be \label{e:op2}
Y^n_t \hat{P}_t= \int_0^t \left\{Y^n_s \hat{V}_s+\chf_{[s\leq S_n]}(\bbE[P_s
  b(s,X_s)|\cG_s]+\hat{\theta}_s)\right\} \,ds +n^2_t,
\ee
where  $n^2$ is a $(\bbP,\cG)$-local martingale.  Equating (\ref{e:op1}) to
(\ref{e:op2}) we get
 \[
\left(\int_0^{t\wedge S_n} \left\{\Phi_s+\hat{P}_s\bbE[b(s,X_s)|\cG_s]-\hat{\theta}_s-\bbE[P_s
  b(s,X_s)|\cG_s]\right\}\,ds\right)_{t \in [0,T]}
\]
is a local martingale, thus, it must vanish since it is also
predictable. We can in fact do similar calculations, after
moving the origin from $0$ to $r \in [0,T]$, to conclude that
 \[
\left(\int_r^{t\wedge S_n} \left\{\Phi_s+\hat{P}_s\bbE[b(s,X_s)|\cG_s]-\hat{\theta}_s-\bbE[P_s
  b(s,X_s)|\cG_s]\right\}\,ds\right)_{t \in [r,T]}
\] must vanish for every $r\geq 0$. Therefore, since $S_n \uar T, \bbP$-a.s.,
\[
\Phi_t=\hat{\theta}_t+\bbE[P_t
  b(t,X_t)|\cG_t]-\hat{P}_t\bbE[b(t,X_t)|\cG_t], \qquad t \in [0,T].
\]
Now, we return to determine $\zeta$.  However, the filtering formula 4.10.8 in \cite{lstm} yields
\[
\zeta_t=\hat{P}_{t-} - v_t,
\]
for some $\cG$-predictable process $v$, which is the unique $\cG$-predictable process satisfying
\be \label{e:v}
\bbE\left[\int_0^{T}\nu_t P_t dD_t\right]=\bbE\left[\int_0^{T}\nu_t v_t dD_t\right],
\ee
for any bounded  $\cG$-predictable  process $\nu$.  We will next show that $v=\left(\bbE[P_t|\cF^Y_t,\tau=t]\right)_{t \in [0,T]}$. Before showing that the candidate process satisfies (\ref{e:v}), let's first verify that it is $\cG$-predictable.

In view of Lemma \ref{l:pAd}, there exist appropriately measurable
functions, $f^1$ and $f^2$ such that the $(\bbQ,
\cF^{Y,\tau})$-optional projections\footnote{$\cF^{Y,\tau}$ is the smallest filtration satisfying
  the usual conditions and including $\cF^Y$ such that $\sigma(\tau)
  \subset \cF^{Y,\tau}_0$.} of $PM$ and $M$ are given by
$(f^1(\tau(\om),\om, t))_{t \in [0,T]}$ and $(f^2(\tau(\om),\om, t))_{t
  \in [0,T] }$, respectively. On the other hand, the Bayes' formula
yields for any $\cF^{Y,\tau}$-stopping time $S$,
\[
\bbE[P_S|\cF^Y_S, \tau]=\frac{\bbE^{\bbQ}[P_S  M_S
  |\cF^Y_S,\tau]}{\bbE^{\bbQ}[M_S|\cF^Y_S,\tau]}=\frac{f^1(\tau(\om),\om,S)}{f^2(\tau(\om),\om,S)}\,;
\]
i.e., $(\bbP, \cF^{Y,\tau})$-optional projection of $P$ is given by
$(f(\tau(\om),\om,t))_{t \in [0,T]}$, where
\[
f(u,\om, t):=\frac{f^1(u,\om,t)}{f^2(u,\om,t)},
\]
for $\om \in \Om$ and $u \geq 0,\, t\geq 0$. Note that $(f(\tau(\om),\om,t))_{t \in [0,T]}$ is $\cF^{Y,\tau}$-optional since $(f^i(\tau(\om),\om,t))_{t \in [0,T]}$ is $\cF^{Y,\tau}$-optional for $i=1,2$.

Moreover, since $f^i(t,\om,t)$ is  $\cF^Y_t$-measurable for $i=1,2$
by Lemma  \ref{l:pAd}, we see that $f(t, \om, t)$ is
$\cF^Y_t$-measurable  for each $t \geq 0$, as well.
Writing
$\bbE[P_t|\cF^Y_t,\tau=t]$ for $f(t,\om, t)$, one has that
$\left(\bbE[P_t|\cF^Y_t,\tau=t]\right)_{t \in [0,T]}$ is a measurable
and $\cF^Y$-adapted process. By the definition of optional
projections,  the $(\bbP,\cF^Y)$-optional projection of $\left(\bbE[P_t|\cF^Y_t,\tau=t]\right)_{t \in [0,T]}$, denoted with
$u$, satisfies $u_t=\bbE[P_t|\cF^Y_t,\tau=t]$
for every $t$. This implies that we can choose an $\cF^Y$-optional
version of $\left(\bbE[P_t|\cF^Y_t,\tau=t]\right)_{t \in
  [0,T]}$. However, since $Y$ is a Brownian motion after an equivalent
change of measure, optional and predictable $\sigma$-algebras coincide yielding the $\cF^Y$-predictability of $\left(\bbE[P_t|\cF^Y_t,\tau=t]\right)_{t \in
  [0,T]}$. Since $\cF^Y$
is a sub-filtration of $\cG$, the claim follows.

Now let's return to verify that $u$,  the $\cF^Y$-predictable (equivalently, $\cF^Y$-optional)
version of $(\bbE[P_t|\cF^Y_t, \tau=t])_{t\in [0,T]}$,  satisfies (\ref{e:v}). Note that since $\cF^Y$ is
contained in $\cF^{Y,\tau}$, $u$ is $\cF^{Y,\tau}$-optional as well. Furthermore,
\bean
\bbE\left[\int_0^{T}\nu_tP_tdD_t\right]&=&-\bbE\left[\nu_{\tau}
  P_{\tau}\chf_{[\tau \leq T]}\right]=-\bbE\left[\chf_{[\tau \leq T]}\nu_{\tau}
  \bbE\left[P_{\tau}\big|\cF^{Y,\tau}_{\tau}\right]\right] \\
&=&-\bbE\left[\chf_{[\tau \leq T]}\nu_{\tau}
  f(\tau,\om,\tau)\right]=\bbE\left[\int_0^{T}\nu_tf(t,\om,t)dD_t\right]\\
&=& \bbE\left[\int_0^{T}\nu_t u_t dD_t\right],\eean
where the third equality follows from the definition of optional
projections and the last equality holds since $u$ is also
$\cF^{Y,\tau}$-optional and a version of $(f(t,\om,t))_{t \in
  [0,t]}$.   This concludes the proof.
\end{proof}

An immediate corollary to this theorem is the following result.
\begin{corollary} \label{c:mainfilter}  Let $P$ defined by (\ref{e:P}) satisfy Assumption
  \ref{a:P} with $P_{\tau}=0$. Then
\[
\hat{P}_t=\hat{P}_0+\int_0^t \hat{V}_s\, ds +\int_0^t\left\{\hat{\theta}_s+\bbE[P_s
b(s,X_s)|\cG_s]-\hat{P}_s\bbE[b(s,X_s)|\cG_s]\right\}\,
  d\beta_s+\int_0^t\hat{P}_{s-}\,dL_s.
\]
\end{corollary}
Note that $P$ vanishes at $\tau$ if $P=f(X)$ where $f$ is a function
that vanishes at $0$. In view of this observation we will next
establish a version of {\em Kushner-Stratonovich equation} (see
Chap.~3 of \cite{bc} for the background) for the
conditional distribution of $X$. To this end let $\bbC$ denote the class
of continuous functions and $\bbC^2_{K,+}$ denote
the class of twice continuously differentiable functions with a
compact support in $(0,\infty)$ and define the operator
$\mathcal{A}:\bbC^2_{K,+} \mapsto \bbC$ by
\[
\mathcal{A}f(x)=a(x)f'(x)+\frac{1}{2}f''(x).
\]
For any $f \in \bbC^2_{K,+}$ let
\[
\pi_t f:=\bbE[f(X_{t\wedge \tau})|\cG_t].
\]
Observe that $\pi_t$ gives the $\cG$-conditional distribution of $X_t$
on the set $[\tau>t]$.
Then, as an immediate corollary to
Corollary \ref{c:mainfilter}, we have the following
\begin{corollary} Let $f \in \bbC^2_{K,+}$. Then,
\be \label{e:K-S}
\pi_t f=\pi_0 f +\int_0^t \pi_s\cA f\, ds +\int_0^t \left\{\pi_s
    fb-\pi_sf\pi_sb\right\}d\beta_s +\int_0^t\pi_{s-}fdL_s.
\ee
\end{corollary}

In particular, if $P$ is the $(\bbP,\cH)$-martingale defined by
$P_t=\bbP[\tau>T|\cH_t]=\chf_{[\tau>t]}H^a(T-t,X_t)$, where $H^a$ is the
function defined in (\ref{d:H}), then $P_{\tau}=0$, too.
We also have that $\hat{P}_{t-}=D_{t-}\hat{P}_t$. Indeed, since
\[
\bbE[ D_s H^a(T-s,X_s)|\cG_s]= D_s \frac{\bbE[ D_s H^a(T-s,X_s)|\cF^Y_s]}{Z_s}
\]
we have that
\[
\lim_{s \uar t} \bbE[ D_s H^a(T-s,X_s)|\cG_s]=\frac{D_{t-}}{Z_t} \lim_{s \uar t}\bbE[ D_s H^a(T-s,X_s)|\cF^Y_s].
\]
However, $(\bbE[ D_s H^a(T-s,X_s)|\cF^Y_s])_{s \in [0,T]}$ is a bounded
$(\bbP,\cF^Y)$-martingale, therefore it is continuous by Theorem 8.3.1 in \cite{Kal}  implying
\[
\lim_{s \uar t} \bbE[ D_s H^a(T-s,X_s)|\cG_s]=\frac{D_{t-}}{Z_t} \bbE[ D_t H^a(T-t,X_t)|\cF^Y_t].
\]
Hence, in view of the corollary above, one can write
\bea
\bbP[\tau>T|\cG_t]&=&\bbE[D_t H^a(T-t,X_t)|\cG_t] \nn  \\
&=&\bbP[\tau>T]\nn\\
&+&\int_0^t \chf_{[s\leq \tau]}\left\{\bbE[H^a(T-s,X_s)b(s,X_s)|\cG_s]-\bbE[H^a(T-s,X_s)|\cG_s]\bbE[b(s,X_s)|\cG_s]\right\}d\beta_s\nn \\
&+&\int_0^t \chf_{[s\leq \tau]}\bbE[H^a(T-s,X_s)|\cG_s]dL_s \label{e:condprob}
\eea
Note that the above formula also gives us the price of a defaultable
zero-coupon bond which pays $1$ unit of a currency to the holder at time-$T$ in case default does not occur, and pays nothing if default does
occur by time-$T$. As discussed in the introduction, there is usually
a rebate paid to the bond holder in case of default. Let's suppose that the rebate is
random and
amounts to $P_{\tau}$ for some stochastic process $P$. Time-$t$ value
of the rebate is given by $\bbE[P_{\tau}\chf_{[\tau \leq T]}|\cG_t]$.
The next proposition gives us a decomposition for the value of the
rebate before default happens.
\begin{proposition} \label{p:rebate} Let $P$ defined by (\ref{e:P}) be
  bounded and
  satisfy Assumption \ref{a:P}. Then, $(\bbE[P_{\tau}\chf_{[t<\tau
    \leq T]}|\cG_t] )_{t \in [0,T]}$ has the unique Doob-Meyer decomposition
\[
\bbE[P_{\tau}\chf_{[t<\tau \leq T]}|\cG_t]=\bbE[\alpha_T|\cG_t]-\alpha_t,
\]
where
\[
\alpha_t=\int_0^{t\wedge
    \tau}\bbE[P_s|\cF^Y_s,\tau=s]\lambda_s\,ds, \qquad t\in[0,T].
\]
\end{proposition}
\begin{proof} Let
\[
R_t:=\bbE[P_{\tau}\chf_{[t<\tau \leq T]}|\cG_t]=\bbE[P_{\tau}\chf_{[\tau\leq
  T]}|\cG_t]-\bbE[P_{\tau}\chf_{[\tau\leq t]}|\cG_t].
\]
Then, $R_T=0$ and
\[
R_t=\left(\bbE[P_{\tau}^+\chf_{[\tau\leq
    T]}|\cG_t]-\bbE[P^+_{\tau}\chf_{[\tau\leq
    t]}|\cG_t]\right)-\left(\bbE[P_{\tau}^-\chf_{[\tau\leq
    T]}|\cG_t]-\bbE[P^-_{\tau}\chf_{[\tau\leq t]}|\cG_t]\right),
\]
where $x^+$ (resp.~$x^-$) denotes the positive (resp.~negative) part
of a real number $x$. The above implies $R$ is the difference
of two positive supermartingales,  thus, by Theorem 8 in
Chap.~III of \cite{Pro}, there exists a predictable process, $\alpha$, of finite variation with $\alpha_0=0$
such that $R-\alpha$ is a $(\bbP,\cG)$-martingale. Since $R_T=0$, we thus have the
unique decomposition of $R$ as follows:
\be \label{e:vdcomp}
R_t=\bbE[\alpha_T|\cG_t]-\alpha_t.
\ee
On the other hand, if we apply integration by parts formula to
$D\hat{P}$ we obtain
\be \label{e:ribp1}
d(D\hat{P})_t=D_{t-}\left\{\hat{V}_t-\bbE[P_t|\cF^Y_t,\tau=t]\lambda_t\right\}dt
+ dn^1_t,
\ee
where $n^1$ is $(\bbP,\cG)$-local martingale. Moreover, since
\[
d(DP)_t=D_tV_t dt + D_t dm_t-P_{t-}dD_t=D_{t-}V_t dt +D_{t-}dm_t -P_{\tau}\chf_{[\tau\leq t]},
\]
by taking the optional projection of the above, we see that
\be \label{e:ribp2}
D_t\hat{P}_t=\hat{P}_0 +\int_0^t
D_{s-}\hat{V}_s\,ds-\bbE[P_{\tau}\chf_{[\tau\leq t]}|\cG_t] +n^2_t,
\ee
where $n^2$ is a $(\bbP,\cG)$-local martingale. Therefore, comparing
(\ref{e:ribp1}) to (\ref{e:ribp2}), we obtain that
\[
\left(\bbE[P_{\tau}\chf_{[\tau\leq t]}|\cG_t] -\int_0^{t\wedge
    \tau}\bbE[P_s|\cF^Y_s,\tau=s]\lambda_s\,ds\right)_{t \in [0,T]}
\]
is a $(\bbP,\cG)$-local martingale. This implies, in view of
$(\bbE[P_{\tau}\chf_{[\tau\leq T]}|\cG_t])_{t\in[0,T]}$ being a
$(\bbP,\cG)$-martingale, that the process $\alpha$ in (\ref{e:vdcomp}) is given by
\[
\alpha_t=\int_0^{t\wedge
    \tau}\bbE[P_s|\cF^Y_s,\tau=s]\lambda_s\,ds, \qquad t\in[0,T].
\]
The claim now follows directly from (\ref{e:vdcomp}).
\end{proof}

We will next look at some specific examples where the finite variation
part in the decomposition of the rebate is of a simpler form.
\begin{example}
In  many situations the rebate is $\cF^Y$-adapted. In this case,
\[
\alpha_t =\int_0^{t\wedge
    \tau}P_s\lambda_s\,ds.
\]
If one is not interested in the Doob-Meyer decomposition but merely
the value of the rebate, it is well known
(see Proposition 5.1.1 in \cite{br}) that
\[
\bbE[P_{\tau}\chf_{[t<\tau \leq
  T]}|\cG_t]=\chf_{[\tau>t]}\frac{1}{Z_t}\bbE\left[-\int_t^TP_u\,dZ_u\bigg|\cF^Y_t\right]=\chf_{[\tau>t]}\frac{1}{Z_t}\bbE\left[\int_t^TP_u\lambda_u
  Z_u\,du\bigg|\cF^Y_t\right].
\]
Recall from Corollary \ref{c:Zrep} that
$Z_t=\exp\left(-\int_0^t\lambda_s\,ds\right) \xi_t^{-1}\kappa_t$ where
$\xi$ and $\kappa$ are as defined in the same corollary.  If we further assume
the condition of Corollary \ref{c:cn}, then there exists a probability
measure $\tilde{\bbQ}\sim \bbP$ such that $\frac{d\tilde{\bbQ}}{d\bbP}=
\xi_T^{-1}\kappa_T$ so that
\[
\bbE[P_{\tau}\chf_{[t<\tau \leq
  T]}|\cG_t]=\chf_{[\tau>t]}\bbE^{\tilde{\bbQ}}\left[\int_t^T
    P_u\lambda_u\exp\left(-\int_t^u
      \lambda_s\,ds\right)\,du\bigg|\cF^Y_t\right],
\]
which agrees with Proposition 4.3 in \cite{cn}.
 The advantage of the formulae above is that they do not contain the
random time $\tau$ inside the expectation on the right hand
side. However, they are valid only if $P$ is $\cF^Y$-adapted.
\end{example}
\begin{example} Similar to the previous example, if the value of
  rebate is given by $F(\tau, Y_{\tau})$ for some deterministic $F$, then
\[
\alpha_t=\int_0^{t \wedge \tau} F(s,Y_s)\lambda_s ds.
\]
\end{example}
The following equation of extrapolation is of interest in its own. Note that the additional assumption that
$p$ defined below is continuous is
automatically satisfied when $\cH$ is a Brownian filtration.
\begin{corollary} Let $P$ defined by (\ref{e:P}) satisfy Assumption
  \ref{a:P}. Fix a $t \in (0,T]$ and set
  $p_s:=\bbE[P_t|\cH_s]$. Assume further that $p$ is continuous. Then,
  for any $s \leq t$
\bean
\bbE[P_t|\cG_s]&=&\bbE[P_t]+\int_0^t\left\{\hat{f}_s
  +\bbE[P_tb(s,X_s)|\cG_s]-\bbE[P_t|\cG_s]\bbE[b(s,X_s)|\cG_s]\right\}d\beta_s\\
&&+\int_0^t\left\{\hat{p}_{s-}-\bbE[p_s|\cF^Y_s, \tau=s]\right\}dL_s,
\eean
where $f$ is the $\cH$-adapted process satisfying $d[p,B]_t=f_t dt$.
\end{corollary}
\begin{proof} Note that $\bbE[P_t|\cG_s]=\hat{p}_s$. Since $p$ is a
  square integrable $(\bbP,\cH)$-martingale, it has an orthogonal
  decomposition of the following form:
\[
p_s=\bbE[P_t|\cH_0]+\int_0^sf_rdB_r + \bar{n}_s
\]
where $\bar{n}$ is a square integrable $(\bbP,\cH)$-martingale orthogonal to
$B $, see Sect.~3 of Chap.~IV in \cite{Pro} . The claim now follows from Theorem \ref{t:mainfilter}.
\end{proof}
\subsection{Extensions}
Acute reader would have noticed that we had not made use of the
Markov property of the vector $(X,Y)$ in the proofs. This makes the
extension of the results of this section to a non-Markovian setting an
easy task.

Indeed, let $\tau$ be an $\cH$-stopping time independent of
the $\cH$-Brownian motion $B$, and the observation process $Y$ is
given by
\be \label{e:extY}
Y_t=B_t +\int_0^t b_s\, ds
\ee
for a progressively measurable process $b$ such that
\[
\bbE\left[\int_0^Tb^2_s\, ds\right]< \infty.
\]
Suppose that
\[
Z_t:=\bbP[\tau>t|\cF^Y_t]=1+\int_0^t\left\{\bbE[\chf_{[\tau>s]}b_s|\cF^Y_s]-Z_s\bbE[b_s|\cF^Y_s]\right\}dB^Y_s
-\int_0^t \lambda_sZ_s\,ds
\]
for some $\cF^Y$-predictable process $\lambda$, where
\[
B^Y_t=Y_t-\int_0^t\bbE[b_s|\cF^Y_s]\,ds
\]
as usual. Then,  all the results of this section
will
continue to hold.

On the other hand, it does not seem easy to relax the assumption that
$\tau$ and $B$ are independent.  The difficulty is not in the
computation of the filtering formulae but the existence of an
absolutely continuous compensator for $Z$, see Remark \ref{r:independence}.

\appendix
\section{Appendix}
\subsection{Proofs of Theorems \ref{t:Hdensity}, \ref{t:AC} and Lemma \ref{l:RN}}
{\sc Proof of Theorem \ref{t:Hdensity}.}\hspace{3mm}Let $\bbQ_x$ denote the law of  the solution of (\ref{e:sdeX}) with the initial condition  $X_0=x$ and $\mathbb{W}_x$ be the law of the standard Brownian motion starting at $x$, both being defined on the canonical space $C(\bbR_+,\bbR)$ where $X_t(\om)=\om(t)$ and $\cF_t=\sigma(X_s; s\leq t)$.
\begin{enumerate}
\item One has,  for any $t\geq 0$,
\[
\bbQ_x|_{\cF_t}=\exp\left(A(X_t)-A(x)-\frac{1}{2}\int_0^t \left\{a^2(X_s)+ a'(X_s)\right\}\,ds\right)\cdot\mathbb{W}_x|_{\cF_t}
\]
The fact that $\exp\left(\int_0^t a(X_s)dX_s-\frac{1}{2}\int_0^t a^2(X_s)\,ds\right)$ is a $(\mathbb{W}_x,\cF)$-martingale follows from the fact that $X$ is the non-exploding solution to (\ref{e:sdeX}) and from, e.g., Exercise 2.10 in Chap.~IX of \cite{RY}. Let $f$ be a test function with a support in $[0,T]$ where $T$ is an arbitrary constant. Then,
\begin{multline}
\bbQ_x[f(\tau)] =\exp\left(-A(x)\right)\mathbb{W}_x \left[f(\tau)\exp\left(A(X_T)-\frac{1}{2}\int_0^T \left\{a^2(X_s)+ a'(X_s)\right\}\,ds\right)\right] \label{e:dens}\\
=\exp\left(-A(x)\right)\mathbb{W}_x \left[\chf_{[\tau\leq T]}f(\tau)\exp\left(A(X_T)-\frac{1}{2}\int_0^T \left\{a^2(X_s)+ a'(X_s)\right\}\,ds\right)\right]\\
=\exp\left(-A(x)\right)\mathbb{W}_x \left[f(\tau)\exp\left(-\frac{1}{2}\int_0^{\tau} \left\{a^2(X_s)+ a'(X_s)\right\}\,ds\right)\right]\\
=\exp\left(-A(x)\right)\mathbb{W}_x \left[f(\tau)\mathbb{W}_x\left[\exp\left(-\frac{1}{2}\int_0^{\tau} \left\{a^2(X_s)+ a'(X_s)\right\}\,ds\right)\bigg| \tau\right]\right],
\end{multline}
where the third equality is due to the Optional Sampling Theorem and
the fact that $f$ vanishes outside $[0,T]$. Since $\tau$ has a
density, namely $\ell(\cdot,x)$, under $\mathbb{W}_x$, we conclude
from the arbitrariness of $T$ that it has a density under $\bbQ_x$ as
well\footnote{Note that we are not claiming that this density
  integrates to $1$; i.e. $\tau$ could be infinite with positive
  $\bbQ_x$-probability. An example of this is when $a\equiv 1$,
  i.e. $X$ is a  Brownian motion with a positive drift.} More precisely,
\[
\bbQ_x[\tau \in dt]=\exp\left(-A(x)\right)\mathbb{E}^{(3)}_x \left[\exp\left(-\frac{1}{2}\int_0^t \left\{a^2(X_s)+ a'(X_s)\right\}\,ds\right)\bigg|X_t=0 \right]\ell(t,x) \,dt,
\]
where  $\mathbb{E}^{(3)}_x$ is expectation with respect to the law of
the 3-dimensional Bessel process starting at $x$. This is due to the
well-known relationship between the law of the Brownian motion
conditioned on its first hitting time of $0$ and that of 3-dimensional
Bessel bridge, which follows from William's time reversal result, see
Corollary 4.6 in Chap.~VII of \cite{RY}). Moreover,
$\mathbb{Q}_x[\tau>0]=1$ since $X$ is continuous and $x>0$. This
proves $H^a(0,x)=1$ and the desired absolute continuity of $H^a$. The strict positivity similarly follows from the fact that $\bbQ_x\sim
  \mathbb{W}_x$, when restricted to $\cF_t$, and that $\mathbb{W}_x[\tau>t]>0$ for every $t\geq
  0$.

\item In order to prove the second claim note that since $a(x) \geq -K_g(1+|x|)$, in view of standard comparison results for the solutions of SDEs (see \cite{RY}), the solution to (\ref{e:sdeX}) is always bigger than the solution of
\[
dX_t=dW_t- K_g(1+|X|_t)\,dt.
\]
Thus, the solution of (\ref{e:sdeX}) is larger than the solution to
\be \label{e:OU}
dX_t=dW_t- K_g(1+X_t)\,dt,
\ee
until the first hitting time of $0$ by the latter.
Let $\bbQ^{(-K_g)}_x$ be the law of the solution of (\ref{e:OU}) with the initial condition $X_0=x$ on the canonical space. Then, by the aforementioned comparison argument we have $\bbQ_x[\frac{1}{\tau} \geq t] \leq Q^{(-K_g)}_x[\frac{1}{\tau} \geq t]$, i.e.
\be \label{e:ctinv} \bbE_x[\frac{1}{\tau}] \leq \bbE^{(-K_g)}_x[\frac{1}{\tau}].
  \ee
Moreover, using the absolute continuity relationship between $Q^{(-K_g)}_x$ and $\mathbb{W}_x$ as above, we obtain
\begin{multline}
\bbQ^{(-K_g)}_x[\tau \in dt]\\
=\exp\left(\frac{K_g}{2}(t + 2 x+  x^2)\right)\mathbb{E}^{(3)}_x \left[\exp\left(-\frac{K_g^2}{2}\int_0^t (1+X_s)^2\,ds\right)\bigg|X_t=0 \right]\ell(t,x) \,dt  \\
\leq\exp\left(\frac{K_g}{2}(t + 2 x+  x^2)\right)\mathbb{E}^{(3)}_x \left[\exp\left(-\frac{K_g^2}{2}\int_0^t(1+ X_s^2)\,ds\right)\bigg|X_t=0 \right]\ell(t,x) \,dt  \\
\leq \frac{2\exp\left(\frac{K_g}{2}t(1-K_g)\right)(K_gt)^{3/2}}{\left[\exp(K_g t/2)-\exp(-K_g t/2)\right]^{3/2}}\exp\left(K_g x-\frac{K_g}{2}x^2\left\{\frac{K_g t coth(K_gt)-1}{K_g t}-1\right\}\right)\ell(t,x) \,dt  \\
=\frac{2\exp\left(-\frac{K_g^2}{2}t\right)(K_gt)^{3/2}}{\left[\exp(K_g t/6)-\exp(-5 K_g t/6)\right]^{3/2}}\exp\left(K_g x-\frac{K_g}{2}x^2\left\{\frac{K_g t coth(K_gt)-1}{K_g t}-1\right\}\right)\ell(t,x) \,dt\\
\leq 2\delta^{3/2} \exp\left(\frac{K_g}{2}(-K_g t + 2 x)\right) \ell(t,x) \,dt, \label{e:d_ubound}
\end{multline}
where the second inequality follows from Formula 2.5 in \cite{Y} and the last line is due to the fact that $\frac{y coth(y)-1}{y} \geq 1$ for $y\geq 0$. Thus,
\[
 \bbE^{(-K_g)}_x\left[\frac{1}{\tau}\right] \leq 2\delta^{3/2}  \exp\left(K_g x\right)\int_0^{\infty}\frac{e^{-\frac{K_g^2}{2} t}}{t} \ell(t,x) \, dt.
\]
Also note that
\[
\int_0^{\infty}\frac{e^{-\frac{K_g^2}{2} t}}{t} \ell(t,x) \, dt=- \frac{\partial }{\partial x} \int_0^{\infty}e^{-\frac{K_g^2}{2} t} \frac{1}{\sqrt{2 \pi t^3}}e^{-\frac{x^2}{2t}} \, dt.
\]
As
\[
\int_0^{\infty}e^{-\frac{K_g^2}{2} t} \frac{1}{\sqrt{2 \pi t^3}}e^{-\frac{x^2}{2t}}\, dt=\frac{1}{x}e^{-K_g x},
\]
Differentiating above with respect to $x$ in conjunction with (\ref{e:ctinv}) yields
\[
\int_0^{\infty} \frac{1}{s}\ell^a(s,x) \,ds \leq 2\delta^{3/2} \frac{1+ K_g x}{x^2}
\]
\item In order to prove the last assertion, first let
\[
\sigma(t,x):=\exp\left(-A(x)\right)\mathbb{E}^{(3)}_x \left[\exp\left(-\frac{1}{2}\int_0^t \left\{a^2(X_s)+ a'(X_s)\right\}\,ds\right)\bigg|X_t=0 \right]
\]
so that $\ell^{a}(t,x)=\sigma(t,x)\ell(t,x)$. Observe that $\sigma$ is
uniformly bounded, locally in $t$, if $A(\infty)> -\infty$. Since $t \ell(t,x)$ is uniformly bounded, there is nothing to prove when $A(\infty)> -\infty$.

When $A(\infty) = - \infty$, we must have $a(\infty)<\infty$. Then, there are two cases to consider: either $a(\infty)>-\infty$ and, consequently, $a$ is bounded on $[0, \infty]$, or $a(\infty)=-\infty$. We will prove the claim in the latter case. The case of bounded $a$ is easier and can be handled by the change of measure technique that we will employ below.

Suppose $a(\infty)=-\infty$ and let $\bbU^k_x$ be the law of the Ornstein-Uhlenbeck process, which is the unique solution to
\[
X_0=x +B_t -k \int_0^t X_s\,ds.
\]
Then, by an application of Girsanov theorem, one has
\begin{multline}
\bbQ_x[\tau \in dt]\\
={\bbU^{K_a}_x}\left[\exp\left(\int_0^t\left\{a(X_s)+K_a X_s\right\}dX_s +\frac{1}{2}\int_0^t \left\{K_a^2 X_s^2 -a^2(X_s)\right\}ds\right)\bigg|X_t=0\right]\bbU^{K_a}_x(\tau\in dt)\\
={\bbU^{K_a}_x}\left[\exp\left(-F(x) +\frac{1}{2}\int_0^t \left\{K_a^2 X_s^2 -a^2(X_s)-a'(X_s)-K_a\right\}ds\right)\bigg|X_t=0\right]\bbU^{K_a}_x(\tau\in dt)\\
\leq K\exp\left(-F(x) \right)\bbU^{K_a}_x(\tau\in dt),
\end{multline}
for some constant $K$, depending on $t$, in view of Assumption \ref{a:a}, where
\[
F(x):=\int_0^x \{a(y)+K_ay\}dy.
\]
Observe that under Assumption \ref{a:a}, for large values of $x$, $\exp(-F(x))\leq \exp( c
x^p) $ for some  constant $c$, and   $p<2$. On the other hand,
\[
t\bbU^{K_a}_x[\tau\in dt]=\frac{x}{\sqrt{2 \pi}}\left(\frac{K_at}{\sinh(K_at)}\right)^{\frac{3}{2}}\exp\left(\frac{K_a}{2}\left(t-x^2(\coth(K_at)-1)\right)\right),
\]
see, e.g. \cite{GY}. Since $\frac{x}{\sinh{x}}$ is bounded and $\coth(K_at) >1$ when $t \leq N$, for any $N$, claim follows.\end{enumerate} \qed

\noindent {\sc Proof of Theorem \ref{t:AC}.}\hspace{3mm}The idea of the
  proof is to apply the nonlinear filtering formulas to find an
  expression for $Z$ which will lead to the statement of the theorem
  after Fubini type arguments as explained below. This will be done in three steps.

\noindent {\sl STEP 1.} We will first prove that \bea
Z_t&=&\bbE[H^a(t, X_0)] \label{e:zfr} \\
&&+\int_0^t
\bbE\left[\chf_{[\tau>s]}H^a(t-s,X_s)\left(b(s,X_s)-\bbE[b(s,X_s)|\cF^Y_s]\right)\big|\cF^Y_s\right]\,
dB^Y_s, \nn
\eea
 and $Z$ is strictly positive. To this
end,  let $P_s:=\chf_{[\tau>s]}H^a(t-s, X_s)$ for $s \leq
  t$. It follows from (\ref{d:H}) and the Markov property of $X$ that, for any $t$,
  $(P_s)_{s \in [0,t]}$ is a bounded, continuous and nonnegative $(\bbP,\cH)$-martingale with
  $P_{\tau}=0$ on the set $[\tau \leq t]$ and
  $P_t=\chf_{[\tau>t]}$ . Since
\be \label{e:spz1}
\int_0^t \bbE^2[b(s,X_s)]\,ds  \leq K_b^2(t) \int_0^t \bbE X_s^2\, ds <
\infty
\ee
in view of Remark \ref{r:2ndmomentX}, it follows from Theorem 8.1 in
  \cite{ls} that for $s \leq t$
\[
\bar{P}_s=\bbE[H^a(t, X_0)] +\int_0^t\left\{
\bbE[P_r b(r,X_r)|\cF^Y_r]-\bar{P}_r\bbE[b(r,X_r)|\cF^Y_r]\right\}\, dB^Y_r,
\]
where $\bar{P}$ is the $\cF^Y$-optional projection of $P$, and the { innovation process} defined by \[
dB^Y_t=Y_t-\bbE[b(t,X_t)|\cF^Y_t],
\]
is an $\cF^Y$-Brownian motion.  Noticing that $Z_t=\bar{P}_t$ yields
the claimed representation.

In order to show the strict positivity we will make use of the process
$M$ defined in  (\ref{e:M}). Observe from the discussion
following (\ref{e:M}) that
 $M^{-1}$ is a strictly positive $(\bbP, \cH)$-martingale, and
 $\bbQ_t\sim \bbP_t$ is a
probability measure on $\cH_t$ defined by
\[
\frac{d\bbQ_t}{d\bbP_t}=M_t^{-1},
\]
under which  $(Y_s)_{s \in [0,t]}$ is a standard Brownian motion
independent of $(X_s)_{s \in [0,t]}$. Also observe that  the laws of $(X_s)_{s \in
  [0,t]}$ under $\bbP_t$ and $\bbQ_t$ are the same since the measure
change only affects $Y$. Moreover, in view of (\ref{e:ZQ}), one has
\[
Z_t \bbE^{\bbQ_t}[M_t|\cF^Y_t]=\bbE^{\bbQ_t}[\chf_{[\tau>t]}M_t|\cF^Y_t].
\]
Since $M$ is strictly positive, so is $\bbE^{\bbQ_t}[M_t|\cF^Y_t]$;
thus, strict positivity of $Z$ is equivalent to that of
$\bbE^{\bbQ_t}[\chf_{[\tau>t]}M_t|\cF^Y_t]$. However, for any $A \in
\cF^Y_t$ with $\bbQ_t[A]>0$, $\bbQ_t[A, \tau>t]=\bbQ_t[A]\bbQ_t[\tau>t] >0$ since $\chf_{[\tau >t]}$ is independent of
$\cF^Y_t$ under $\bbQ_t$, and $\bbQ_t[\tau>t]=\bbP[\tau>t]>0$ in view
of Part 1 of Theorem \ref{t:Hdensity}. Thus,
$\bbE^{\bbQ_t}[\chf_{[\tau>t]}M_t|\cF^Y_t]>0,\, \bbQ_t$-a.s. since
$M_t$ is strictly positive $\bbQ_t$-a.s.. Claim now follows from
the equivalence of $\bbQ_t$ and $\bbP_t$.

{\sl STEP 2.} Next, we will show that
\bea
&&\int_0^t
\bbE\left[\chf_{[\tau>s]}H^a(t-s,X_s)\left(b(s,X_s)-\bbE[b(s,X_s)|\cF^Y_s]\right)\big|\cF^Y_s\right]\,
dB^Y_s \nn \\
&& =\int_0^t\bbE\left[\chf_{[\tau>s]}\left(b(s,X_s)-\bbE[b(s,X_s)|\cF^Y_s]\right)|\cF^Y_s\right]dB^Y_s \label{e:fubos}\\
&&-\int_0^t\left(\int_0^s \bbE[\chf_{[\tau>r]}\ell^a(s-r, X_r)\left(b(r,X_r)-\bbE[b(r,X_r)|\cF^Y_r]\right)|\cF^Y_r]dB^Y_r\right)ds. \nn
\eea
 Recall that
\[
H^a(t-s,X_s)=1-\int_s^{t}\ell^a(u-s,X_s)\,
du;
\]
thus,
\bean
&&\int_0^t
\bbE\left[\chf_{[\tau>s]}H^a(t-s,X_s)\left(b(s,X_s)-\bbE[b(s,X_s)|\cF^Y_s]\right)\big|\cF^Y_s\right]\,
dB^Y_s\\
&=&
\int_0^t\bbE\left[\chf_{[\tau>s]}\left(b(s,X_s)-\bbE[b(s,X_s)|\cF^Y_s]\right)|\cF^Y_s\right]dB^Y_s
\\
&&-\int_0^t\bbE\left[\chf_{[\tau>s]}\int_s^t
  \ell^a(u-s,X_s)\,du\left(b(s,X_s)-\bbE[b(s,X_s)|\cF^Y_s]\right)\bigg|\cF^Y_s\right]dB^Y_s\\
&=&\int_0^t\left\{\bbE\left[\chf_{[\tau>s]}b(s,X_s)|\cF^Y_s\right]-Z_s\bbE[b(s,X_s)|\cF^Y_s]\right\}dB^Y_s
\\
&&-\int_0^t \int_s^t\bbE\left[\chf_{[\tau>s]}
  \ell^a(u-s,X_s)\left(b(s,X_s)-\bbE[b(s,X_s)|\cF^Y_s]\right)\bigg|\cF^Y_s\right]\,du\,
dB^Y_s
\eean
where the interchange of expectation and integration is justified by
Fubini's theorem since $\ell^a$ is positive and integrable, $b$ is
Lipschitz, and $\bbE|X_s|<\infty$ for any $s \geq 0$.

Moreover, if we can interchange the order of stochastic and ordinary
integrals in the second integral above, we can further write
\bean
&&\int_0^t
\bbE\left[\chf_{[\tau>s]}H^a(t-s,X_s)\left(b(s,X_s)-\bbE[b(s,X_s)|\cF^Y_s]\right)\big|\cF^Y_s\right]\,
dB^Y_s\\
&=&\int_0^t\eta_s\,dB^Y_s -\int_0^t\left(\int_0^u \bbE\left[\chf_{[\tau>s]}\ell^a(u-s, X_s)\left(b(s,X_s)-\bbE[b(s,X_s)|\cF^Y_s]\right)\big|\cF^Y_s\right]dB^Y_s\right)du.
\eean
This interchange of ordinary and stochastic integrals can be justified by Theorem 65 in Chap.~IV of
\cite{Pro} if
\be \label{e:fj}
\bbE\left[\int_0^t\int_0^s\chf_{[\tau>r]}(\ell^a(s-r,
    X_r))^2 X_r^2 \, dr\, ds\right] < \infty
\ee
since $b$ is locally Lipschitz and $b(t,0)=0$ by Assumption \ref{a:b}. Since all the terms are positive we have
\bean
\int_0^t\int_0^s\chf_{[\tau>r]}(\ell^a(s-r,
    X_r))^2 X_r^2 \, dr\, ds&=&\int_0^t\int_r^t\chf_{[\tau>r]}(\ell^a(s-r,
    X_r))^2 X_r^2  \, ds\, dr \\
&\leq& K \int_0^t\chf_{[\tau>r]} X_r^2 \int_r^t\frac{1}{s-r}\ell^a(s-r,
    X_r) \, ds\, dr\\
 &\leq& K \int_0^t\chf_{[\tau>r]} X_r^2  \int_r^{\infty}\frac{1}{s-r}\ell^a(s-r,
    X_r) \, ds\, dr\\
&\leq& K \int_0^t\chf_{[\tau>r]}  (1+K_g X_r) dr\\
&\leq& K \int_0^t (1+K_g |X_r|)dr.
\eean
where the second line is due to $\ell(u, x) < K \frac{1}{u}$ by
Theorem \ref{t:Hdensity} for some constant $K$, possibly depending on
$t$,  and the fourth line is a consequence of (\ref{e:tauinv}).
(\ref{e:fj}) now follows from Remark \ref{r:2ndmomentX}.

\noindent {\sl STEP 3.} Combining (\ref{e:zfr}) and (\ref{e:fubos})  yields
\bean
Z_t&=&\bbE[H^a(t,X_0)]+\int_0^t \eta_s\,dB^Y_s \\
&-&\int_0^t\left(\int_0^s
  \bbE[\chf_{[\tau>r]}\ell^a(s-r,
  X_r)\left(b(r,X_r)-\bbE[b(r,X_r)|\cF^Y_r]\right)|\cF^Y_r]dB^Y_r\right)ds.
\eean
The proof is now complete since
\[
\bbE[H^a(t,X_0)]=1-\int_0^t\int_0^{\infty}
\ell^a(u,x)\mu(dx)\,du.
\]
\qed

\noindent {\sc Proof of Lemma \ref{l:RN}.}\hspace{3mm}   Observe that for any bounded $\cF^Y_t$-measurable
  random variable $F$, we can write,  in view of the absolute continuity relationship between $\bbP$
  and $\bbQ$,
\[
\bbE[F]=\bbEQ\left[M_t F\right]=\bbEQ\left[\bbEQ[M_t|\cF^Y_t]F\right].
\]
Since $(\bbEQ[M_t|\cF^Y_t])_{t \in [0,T]}$ is a strictly positive
$(\bbQ,\cF^Y)$-martingale, the above implies that
\[
d\bbP|_{\cF^Y_t}=\bbEQ[M_t|\cF^Y_t]d\bbQ|_{\cF^Y_t},
\]
and $\bbP|_{\cF^Y_t}\sim\bbQ|_{\cF^Y_t}$. Moreover, since
$Y$ is $\bbQ$-Brownian motion, we have from the predictable
representation property of Brownian filtrations that
\[
\bbEQ[M_t|\cF^Y_t]=1+\int_0^t \phi_s \bbEQ[M_s|\cF^Y_s]\,dY_s,
\]
for some $\cF^Y$-predictable process $\phi$ since
$(\bbEQ[M_s|\cF^Y_s])_{s \in [0,T]}$ is strictly positive and
continuous, hence predictable.
Next note that
$\cF^Y$-canonical decomposition of $Y$ under $\bbP$ is given by
\[
Y_t=B^Y_t+\int_0^t\bbE[b(s,X_s)|\cF^Y_s]\,ds
\]
by an application of Theorem 8.1 in \cite{ls}. Girsanov Theorem now
tells us that $\phi_s=\bbE[b(s,X_s)|\cF^Y_s]$, i.e.
\[
\bbEQ[M_t|\cF^Y_t]
=1+\int_0^t
\bbEQ[M_s|\cF^Y_s]\bbE[b(s,X_s)|\cF^Y_s]\,dY_s.
\]
Moreover, since $\bbE[b(s,X_s)|\cF^Y_s]=\frac{\bbEQ[M_s
  b(s,X_s)|\cF^Y_s]}{\bbEQ[M_s|\cF^Y_s]}$, we also have
\[
\bbEQ[M_t|\cF^Y_t]
=1+\int_0^t
\bbEQ[M_sb(s,X_s)|\cF^Y_s]\,dY_s.
\]
In view of (\ref{e:minv}), an
application of Ito's formula yields that
\[
\bbE[M_t^{-1}|\cF^Y_t]=1-\int_0^t
\bbE[M_s^{-1}|\cF^Y_s]\bbE[b(s,X_s)|\cF^Y_s]\,dB^Y_s.
\]
\qed

\subsection{A measure theoretic lemma}
In this section we will state and prove a lemma which will be useful in
obtaining the main filtering result of this paper contained in Theorem
\ref{t:mainfilter}. The proof is based
on elementary measure theoretic methods. We refer the reader to
Section 5 in Chap. IV of \cite{RY} for equivalent definitions of
optional projections, which will be used in the proof below. In Lemma below $\bbQ$ is the probability measure on $\cH_{\tau \vee T}$ which is equivalent\footnote{As before $\bbQ$ is defined via the Radon-Nikodym derivative $M_{\tau \vee T}$. Observe that $M$ is still a martingale until the finite stopping time $\tau \vee T$ in view of the same no-explosion argument used in the beginning of the proof of Theorem \ref{t:Hdensity}. Also recall from the discussion following the definition of $M$ in (\ref{e:M}) the impossibility of defining an equivalent $\bbQ$ on $\cH_{\infty}$. } to the restriction of $\bbP$ to $\cH_{\tau \vee T}$ and under which $Y^{\tau \vee T}$ is a stopped Brownian motion independent of $X^{\tau \vee T}$.  In what follows $\cB$ denotes the class of Borel sets and we suppress the dependency on $T$ to ease notation when no confusion arises.

 \begin{lemma} \label{l:pAd}
\begin{enumerate}
\item  Let  $T>0$ be a fixed real number and suppose that $F$ is a $\cB([0,T])\otimes (\sigma(\tau)\vee \cH_T)$-measurable and $\bbQ$-integrable stochastic process.  Denote  the $(\cF^Y_t)_{t \in [0,T]}$-optional
  $\sigma$-algebra with $\cO^Y$  and let $\cF^{Y,\tau}$ be the smallest filtration satisfying
  the usual conditions and including $(\cF^Y_t )_{t \in [0,T]}$ such that $\sigma(\tau)
  \subset \cF^{Y,\tau}_0$. Then, there exists a function
  $f:[0,\infty)\times(\Om\times [0,T)) \mapsto \bbR$ such that $f
 $ is $\cB([0,\infty)) \otimes \cO^Y$-measurable and the $(\bbQ, \cF^{Y,\tau})$-optional
 projection of $F$ is given by $(f(\tau(\om),\om, t))_{t \in [0,T]}$.
\item For every $u \geq 0$ and $t\geq 0$, $\bbE^{\bbQ}[F_t|\cF^Y_t,\tau=u] :=f(u, \cdot, t)$ is $\cF^Y_t$-measurable,
  where $f$ is as above.
\end{enumerate}
\end{lemma}
\begin{proof}
\begin{enumerate}
\item Note that we can assume without any loss of
  generality that $F$ is $\cB([0,T)) \otimes
  (\sigma(\tau)\vee \cF^X_{T}\vee\cF^Y_T)$-measurable in view of the tower
  property of conditional expectations.
 Let $\cI$ and $\cC$ denote the class of $\bbQ$-integrable and $\cB([0,T)) \otimes
  (\sigma(\tau)\vee \cF^X_{T}\vee\cF^Y_{T})$-measurable stochastic processes, and
  the class of $\cB([0,\infty)) \otimes
  \cO^Y$-measurable real-valued functions defined on
  $[0,\infty)\times(\Om\times [0,T))$, respectively. For $F \in \cI$, let's denote
 its $(\bbQ,\cF^{Y,\tau})$-optional projection with ${}^o F$
  and define
\[
\cR:=\left\{F \in \cI: {}^oF=(f(\tau(\om),\om,t))_{t \in [0,T]}, \, f \in  \cC\right\}.
\]
$\cR$ is clearly a vector space containing constant functions. Moreover,
if $(F^n)_{n \geq 1} \subset \cR$ is a sequence of uniformly bounded and increasing
processes such that  $\lim_{n \rar \infty}
F^n=F$, then $F \in \cR$, as well. Indeed, let $f^n$ denote
the measurable function corresponding to ${}^oF^n$ for each $n$. Then,
$f:=\liminf_{n \rar \infty} f^n$ belongs to $\cC$ since $f^n \in \cC$ for each $n$. Moreover, for any
$\cF^{Y,\tau}$-stopping time $S$, which is necessarily less than or equal to $T$, 
\bean
f(\tau(\om), \om, S(\om))&=&\liminf_{n
  \rar \infty} f^n(\tau(\om), \om, S(\om))\\
&=&\liminf_{n
  \rar \infty}\bbE^{\bbQ}\left[F^n_S\big|\cF^{Y,\tau}_S\right]\\
&=&\bbE^{\bbQ}\left[F_S\big|\cF^{Y,\tau}_S\right],
\eean
where the second equality follows from the definition of optional
projections and the last equality follows from the Dominated
Convergence Theorem. This shows  that
${}^oF=(f(\tau(\om),\om,t))_{t \in [0,T]}$ and, thus, $F \in
\cR$. Consequently, $\cR$ is a monotone vector space.

In order to prove the claim using a monotone class argument, it
suffices to prove the statement for  a multiplicative class generating
$\cI$. Such a class is provided by the processes
\[
F_t(\om)=\chf_{[0,s)}(t) F^1(\om) F^2(\om) F^3(\tau(\om)),  \qquad 0 \leq s \leq
T, \qquad F^1 \in L^{\infty}(\cF^X_{T}),\;  F^2
\in L^{\infty}(\cF^Y_{T}) 
\]
and $F^3$ is a bounded Borel measurable function on $[0,\infty)$. Let $(f^2(\om,t))_{t \in [0,T]}$ be the \cadlag version of the $(\bbQ,
\cF^Y)$-martingale $(\bbE^{\bbQ}[F^2|\cF^Y_t])_{t \in [0,T]}$ and note that $f^2$ is an
$\cO^Y$-measurable function. Moreover, $(f^2(\om,t))_{t \in [0,T]}$  is
also a $(\bbQ,
\cF^{Y,\tau})$-martingale since $(Y_t)_{t \in [0,T]}$ and $\tau$ are independent under
$\bbQ$. Therefore, the $(\bbQ,
\cF^{Y,\tau})$-optional projection of $F^2$ is given by $f^2$ in view
of the Optional Stopping Theorem.

Also observe that there exists a Borel measurable function, $f^1$,
such that $f^1(\tau)=\bbE^{\bbQ}[F^1|\tau]$. We will now see that
${}^o F =(\chf_{[0,s)}(t) f^1(\tau(\om))f^2(\om,t)F^3(\tau(\om)))_{t \in [0,T]}$. Clearly,
$(\chf_{[0,s)}(t) f^1(\tau(\om))f^2(\om,t)F^3(\tau(\om)))_{t \in [0,T]}$  is an
$\cF^{Y,\tau}$-optional process since $f^2$ is $\cF^Y$-optional. In order to show it is the desired
optional projection, it suffices to show that for any
$\cF^{Y,\tau}$-stopping time $S$
\[
\bbE^{\bbQ}\left[F_S\right]=\bbE^{\bbQ}\left[\chf_{[0,s)}(S) f^1(\tau)f^2(\om,S)F^3(\tau(\om))\right].
\]
Indeed,
\bean
\bbE^{\bbQ}\left[\chf_{[0,s)}(S)F^1
  F^2F^3(\tau)\right]&=&\bbE^{\bbQ}\left[\chf_{[0,s)}(S)\bbE^{\bbQ}\left[F^1
    \big| \cF^Y_{T},\tau\right] F^2 F^3(\tau)\right]\\
&=&\bbE^{\bbQ}\left[\chf_{[0,s)}(S)\bbE^{\bbQ}\left[F^1
    \big| \tau\right] F^2F^3(\tau)\right]\\
&=&\bbE^{\bbQ}\left[\chf_{[0,s)}(S)f^1(\tau)F^2F^3(\tau)\right]\\
&=&\bbE^{\bbQ}\left[\chf_{[0,s)}(S)f^1(\tau)f^2(\om,S)F^3(\tau)\right],
\eean
where the second equality is due to the independence of $X^{\tau \vee T}$ and $Y^{T}$
under $\bbQ$ and the last equality holds since $f^2$ is the $(\bbQ,
\cF^{Y,\tau})$-optional projection of $F^2$.

Finally, since we have already observed that $f^2$ is an
$\cO^Y$-measurable function, it now easily follows that the function $f(u, \om, t):=
\chf_{[0,s)}(t) f^1(u)f^2(\om,t)F^3(u)$ belongs to $\cC$. The Monotone Class
Theorem now
yields that any bounded member of $\cI$ is contained in
$\cR$. The general case follows from applying the Dominated Convergence
Theorem to $F$ and the sequence $((F \wedge n)\vee -n)$.
\item Note that the $u$-section, $f( u, \cdot, \cdot)$,  of $f$  is
  $\cO^Y$-measurable for each $u \geq 0$ since $f$ is measurable
  with respect to the product $\sigma$-algebra. In particular,  $(f( u,
  \om,t))_{t \in [0,T]}$ is
  $\cF^Y$-adapted for each $u \geq 0$.
\end{enumerate}
\end{proof}
\end{document}